\documentclass{amsart}
\usepackage{latexsym,amsxtra,amscd,ifthen}
\usepackage{amsfonts}
\usepackage{verbatim}
\usepackage{amsmath}
\usepackage{amsthm}
\usepackage{amssymb}
\usepackage{url}
\usepackage[arrow,matrix]{xy}
\usepackage{graphicx}
\usepackage{tikz}
\usepackage{dsfont}
\usepackage{color}

\theoremstyle{plain}

\newtheorem{theorem}{Theorem}
\newtheorem{lemma}[theorem]{Lemma}
\newtheorem{proposition}[theorem]{Proposition}
\newtheorem{corollary}[theorem]{Corollary}

\numberwithin{theorem}{section}
\numberwithin{equation}{theorem}

\theoremstyle{definition}
\newtheorem{definition}[theorem]{Definition}

\newtheorem{example}[theorem]{Example}
\newtheorem{remark}[theorem]{Remark}

\newtheorem*{question*}{Question}

\DeclareMathOperator{\fpdim}{fpd}

\DeclareMathOperator{\Ext}{Ext}

\DeclareMathOperator{\Hom}{Hom}

\xyoption{all}

\begin{document}

\title{Frobenius-Perron theory of the bound quiver algebras containing loops}

\author{J.M.Chen and J.Y.Chen$^*$}

\thanks{$^*$ the corresponding author}
\date{\today}

\begin{abstract}
The {Frobenius-Perron dimension of a matrix, also known as spectral radius}, is a useful tool for studying linear algebras and plays an important role in the classification of the representation categories of algebras.  In this paper, we study the Frobenius-Perron theory of the representation categories of bound quiver algebras containing loops, and find a way to calculate the Frobenius-Perron dimension of these algebras when they satisfy the commutativity condition of loops. As an application, we prove that the Frobenius-Perron dimension of the representation category of a modified ADE bounded quiver algebra is equal to the maximum number of loops at each vertex. Finally, we point out that there also exists infinite dimensional algebras whose Frobenius-Perron dimension is equal to the maximal number of loops by giving an example.
\end{abstract}


\maketitle


\setcounter{section}{0}
\section{Introduction}
\label{xxsec0}



\bigskip

The spectral radius (also called the Frobenius-Perron dimension) of a
matrix is an elementary and extremely useful invariant in linear algebra,
combinatorics, topology, probability and statistics. For instance, we can classify all the finite graphs which are simple and connected by applying the spectral radius to adjacency matrix of them \cite{DG}.

The Frobenius-Perron dimension of an object in a semisimple finite tensor (or
fusion) category was introduced by Etingof-Nikshych-Ostrik in 2005 \cite{ENO} (also see \cite{EG, EGO, N}). Since then
it has become an extremely useful invariant in the study of fusion categories and
representations of semismiple (weak and/or quiasi-)Hopf algebras.

In 2017, the Frobenius-Perron dimension of an endofunctor of a category was
introduced by the authors in \cite{CG1}. It can be viewed as a
generalization of the Frobenius-Perron dimension of an object in
a fusion category introduced by Etingof-Nikshych-Ostrik \cite{ENO}. It was
shown in \cite{CG1, CG2, ZZ} that the Frobenius-Perron dimension has
strong connections with the representation type of a category.

To gain a better understanding of the Frobenius-Perron dimension of an endofunctor, Wicks \cite{W} calculated the Frobenius-Perron dimension of the representation category of a modified ADE bounded quiver algebra with arrows in a certain direction. It showed that the Frobenius-Perron dimension of this category was equal to the maximum number of loops at a vertex, and asked what would happen if the directions of the arrows were changed.

In this paper, we study the Frobenius-Perron theory of the representation categories of the bound quiver algebras containing loops.  As we know that a bound quiver algebra containing loops is representation-infinite and it is hard to describe the homomorphism spaces (or extension spaces) between objects in representation category completely. We focus on these algebras when they satisfy the commutativity condition of loops (see Definition \ref{def3.1}).
Let $A$ be a bound quiver algebras satisfying the commutativity condition of loops and {$B=A/J$ be its quotient algebra where $J$ }is the ideal generated by all the loops. We show that we can describe the extension spaces in $A$ through the one in $B$ in certain extend. Inspired by the work in \cite{CC}, which we  calculated the Frobenius-Perron dimension of the representation categories of representation-directed algebras, we consider the case of $B$ is a representation-directed algebra, and prove that the Frobenius-Perron dimension of $A$ is equal to the maximal number of loops at each vertex. As an application, we confirm that the Frobenius-Perron dimension of the representation category of a modified ADE {bound} quiver algebra is equal to the maximal number of loops at each vertex no matter what the directions of arrows we choose, which {presents} an explicit answer to the question asked in \cite{W}.  We further discuss the case of $B$ is a canonical algebra of type ADE in this paper, and prove that the Frobenius-Perron dimension of $A$ fall in an interval with length less than 1. At last, we show that there also exists infinite dimensional algebras whose Frobenius-Perron dimension is equal to the maximal number of loops by calculating the Frobenius-Perron dimension of {the representation categories of polynomial algebras.}

\subsection{Conventions}
\label{xxsec0.8}
\begin{enumerate}
\item[(1)]
Throughout let $\Bbbk$ be an algebraically closed field, and let everything
be over the field $\Bbbk$.
\item[(2)]
Usually $Q$ means a finite connected quiver.
\item[(3)]
If $A$ is an algebra over the base field $\Bbbk$, then we denote by $A$-mod
the category of finite dimensional left $A$-modules.
\end{enumerate}

\bigskip

The paper is organized as follows. In Section 1, we introduce the background and summarize the main work of this paper. In Section 2, we review the definition of Frobenius-Perron dimension of a $\Bbbk$-linear category.
In Section 3, we study the loop-extended algebras (see Definition \ref{def3.1}) and describe the properties of the extension spaces over the representation categories of these algebras.
In Section 4, we find a way to obtain the Frobenius-Perron dimension of loop-extended algebras of representation-directed algebras which include ADE quiver algebras as special cases. In Section 5, we study the Frobenius-Perron dimension of a tube. In Section 6, we calculate the Frobenius-Perron dimension of loop-extended algebras of canonical algebras and give some examples. In Section 7, we give the Frobenius-Perron dimension of the representation categories of the polynomial algebras. The following two theorems are main results of this paper which are proved in Theorem \ref{xxthm2.3}, Theorem \ref{xxthm3.2} and Theorem \ref{thm6.3}.

\begin{theorem}
\label{thm1.1}
Let $A=\Bbbk Q/I$ be the bound quiver algebra of a finite quiver $Q$, where $I$ is an admissible ideal satisfying the commutativity condition of loops. $B$ is the loop-reduced algebra of $A$ (see Definition \ref{def3.1}). Then the following hold.

$(1)$ If $M,N$ are two $B$-modules with $\Hom_B(M,N)=0$, then \[
\Ext_A^1(M,N)\cong\Ext_B^1(M,N).
\]

$(2)$ If $M$ is a brick $B$-module which is not simple, then \[
\Ext_A^1(M,M)\cong\Ext_B^1(M,M).
\]
\end{theorem}

\begin{theorem}
Keep the notation as in Theorem \ref{thm1.1}. Then the following hold.

$(1)$ If $B$ is representation-directed, then\[
\fpdim{(A{\text -{\rm mod}})}=\max_{P\in Q_0}N_P
\]where $N_P$ is the number of loops at $P$.

$(2)$ If $B$ is a canonical algebra of type {\rm ADE}
, then\[
\fpdim{(A{\text -{\rm mod}})}\in [n_{max},n_{max}+1)
\]where $n_{max}$ is the maximal number of loops at each vertex in the quiver.
\end{theorem}

\section{Preliminaries}
\label{xxsec1}

\subsection{$\Bbbk$-linear categories}
\label{xxsec1.2}
If ${\mathcal C}$ is a $\Bbbk$-linear category, then
$\Hom_{\mathcal C}(M,N)$ is a $\Bbbk$-linear space for all
objects $M,N$ in ${\mathcal C}$. If ${\mathcal C}$ is
also abelian, then $\Ext^i_{\mathcal C}(M,N)$ are
$\Bbbk$-linear spaces for all $i\geq 0$. Let $\dim$ be the
$\Bbbk$-vector space dimension.

Throughout the rest of the paper, let ${\mathcal C}$ denote a
$\Bbbk$-linear category. A functor between two $\Bbbk$-linear
categories is assumed to preserve the $\Bbbk$-linear structure.
For simplicity,
$\dim(M,N)$ stands for $\dim \Hom_{\mathcal C}(M,N)$ for any
objects $A$ and $B$ in ${\mathcal C}$.

The set of finite subsets of nonzero objects in ${\mathcal C}$
is denoted by $\Phi$ and the set of subsets of $n$ nonzero objects
in ${\mathcal C}$ is denoted by $\Phi_n$ for each $n\geq 1$.
It is clear that $\Phi=\bigcup_{n\geq 1} \Phi_n$. We do not
consider the empty set as an element of $\Phi$.

\begin{definition}\cite[Definition 2.1]{CG1}
\label{xxdef2.1}
Let ${\mathcal C}$ be a $\Bbbk$-linear abelian category, $\phi:=\{X_1, X_2, \cdots,X_n\}$ be a finite subset of nonzero
objects in ${\mathcal C}$, namely, $\phi\in \Phi_n$.
\begin{enumerate}
\item[(1)]
The {\it adjacency matrix} of $\phi$ is defined to be
$$C(\phi):=(a_{ij})_{n\times n}, \quad
{\text{where}}\;\; a_{ij}:=\dim\Ext^1_{\mathcal C}(X_i, X_j) \;\;\forall i,j.$$
\item[(2)]
An object $M$ in ${\mathcal C}$ is called a {\it brick}
if
\begin{equation}
\notag
\Hom_{\mathcal C}(M,M)=\Bbbk.
\end{equation}

\item[(3)]
$\phi\in \Phi$ is called a {\it brick set}   if each $X_i$ is a brick
and
$$\dim(X_i, X_j)=\delta_{ij}$$
for all $1\leq i,j\leq n$. The set of brick
$n$-object subsets is denoted by $\Phi_{n,b}$ . We write $\Phi_{b}=\bigcup_{n\geq 1}
\Phi_{n,b}$.
\end{enumerate}
\end{definition}

Let $C$ be an $n\times n$-matrix over complex field ${\mathbb C}$.
The {\it spectral radius} of $C$ is defined to be
$$\rho(C):=\max\{ |r_1|, |r_2|, \cdots, |r_n|\}\quad \in {\mathbb R}$$
where $\{r_1,r_2,\cdots, r_n\}$ is the complete multi-set of eigenvalues of $C$.

\begin{definition}\cite[Definition 2.3]{CG1}
\label{xxdef2.3}
Retain the notation as in Definition \ref{xxdef2.1}, and
we use $\Phi_{b}$ as the testing objects.
\begin{enumerate}
\item[(1)]
The {\it $n$th Frobenius-Perron dimension} of $\mathcal{C}$ is defined to be
$$\fpdim^n (\mathcal{C}):=\sup_{\phi\in \Phi_{n,b}}\{\rho(C(\phi))\}.$$
If $\Phi_{n,b}$ is empty, then by convention, $\fpdim^n(\mathcal{C})=0$.
\item[(2)]
The {\it Frobenius-Perron dimension} of $\mathcal{C}$ is defined to be
$$\fpdim (\mathcal{C}):=\sup_n \{\fpdim^n(\mathcal{C})\}
=\sup_{\phi\in \Phi_{b}} \{\rho(C(\phi)) \}.$$
\end{enumerate}
\end{definition}


\subsection{Representation of bound quivers}

Let $Q$ be a finite connected quiver and $I$ be an admissible ideal of $\Bbbk Q$. Then $(Q,I)$ is called a {\it bound quiver} and $A=\Bbbk Q/I$ is the {\it bound quiver algebra} of $Q$ with respect to $I$. A {\it representation} of $(Q,I)$ is a tuple $M:=(M_a,M_\alpha)_{a\in Q_0,\alpha\in Q_1}$, satisfying

$(R1)$ To each point $a$, $M_a$ is a finite dimensional $\Bbbk$-vector space;

$(R2)$ To each arrow $\alpha:a\rightarrow b$, $M_\alpha$ is a $\Bbbk$-linear map from $M_a$ to $M_b$;

$(R3)$ If $\sum\limits_{i=1}^m \lambda_i\alpha_{i,1}\cdots\alpha_{i,n_i} $ belongs to $I$, where $\lambda_i\in \Bbbk$ (not all zero), then it implies $\sum\limits_{i=1}^m \lambda_i M_{\alpha_{i,1}}\cdots M_{\alpha_{i,n_i}}=0$.\\
Assume $M=(M_a,M_\alpha)$ and $N=(N_a,N_\alpha)$ are two representations of $(Q,I)$, a morphism $f:M\rightarrow N$ from $M$ to $N$ is a tuple $f=(f_a: M_{a}\rightarrow N_{a})_{a\in Q_0}$ of $\Bbbk$-linear maps such that $f_b\circ M_\alpha=N_\alpha\circ f_a$ hold for each arrow $\alpha:a\rightarrow b$.

Denote by rep$(Q,I)$ the category of representations of $(Q,I)$. The following theorem is from \cite{ASS}.

\begin{theorem}
\cite[Ch.\uppercase\expandafter{\romannumeral3}, Theorem 1.6] {ASS}
Let $(Q,I)$ be a bound quiver and $A=\Bbbk Q/I$ be the bound quiver algebra of $Q$ with respect to $I$. There exists a $\Bbbk$-linear equivalence of categories\[
F:{A{\text -{\rm mod}}}\Tilde{\longrightarrow} {\rm rep}(Q,I).
\]
\end{theorem}

\section{Representation category of a bound quiver containing loops}

\begin{definition}
\label{def3.1}
Let $A=\Bbbk Q/I$ be the bound quiver algebra of a finite quiver $Q$, where $I$ is an admissible ideal. We say $I$ \emph{satisfies the commutativity condition of loops} if the following conditions hold.

$(a)$  For a path $\gamma\alpha$, if $\gamma$ is a loop and $\alpha$ is an arrow but not a loop, then  $\gamma\alpha$ belongs to $I$;

$(b)$  For a path $\beta\gamma$, if $\gamma$ is a loop and $\beta$ is an arrow but not a loop, then $\beta\gamma$ belongs to $I$;

$(c)$  For any two loops $\gamma_1,\gamma_2$ base at the same vertex, $\gamma_1\gamma_2-\gamma_2\gamma_1$ belongs to $I$.\\
In this case, let $J$ be the ideal of $A$ generated by all the loops. The quotient algebra $B:=A/J$ is called the {\it loop-reduced algebra} of $A$ and we call $A$ a {\it loop-extended algebra} of $B$.
\end{definition}

In this paper, we only consider bound quiver algebras satisfying the commutativity condition of loops.
{The following proposition is from \cite{W}, for further application, we present a categorical proof here.}

\begin{proposition}(\cite[Proposition 3.4]{W})
\label{xxcor2.2}
Let $A=\Bbbk Q/I$ be the bound quiver algebra of $Q$, where $I$ is an admissible ideal of $\Bbbk Q$ satisfying the commutativity condition of loops. Assume $B$ is the loop-reduced algebra of $A$, then we get a one-to-one correspondence between the isomorphism classes below:\[
\{\text{brick}\ A\text{-modules}\} \leftrightarrow\{\text{brick}\ B\text{-modules}\}.
\]Moreover, for each two brick $B$-modules $M,N$, there exists a natural isomorphism\[
\Hom_A(M,N)\cong\Hom_{B}(M,N).
\]
\end{proposition}

\begin{proof}
Since $B$ is a quotient algebra of $A$, we have $\Hom_{A}(M,M)\cong\Hom_{B}(M,M)=k$ for any brick $B$-module $M$. So \{brick$\ B$-modules\} is a subset of \{brick$\ A$-modules\}. Conversely,  if there exists a brick $A$-module $N$ doesn't belong to the set \{brick$\ B$-modules\}, then we can find a vertex $P$ in $Q$ and a loop $\gamma_0$ at $P$ such that $N_{\gamma_0}\ne0$. Consider the map $f: N\rightarrow N$ defined as follow
\[
f_R=\begin{cases}
    0, R\in Q_{0}\backslash \{P\}\\
    N_{\gamma_0}, R=P.
\end{cases}
\]It is not hard to prove that $f$ is an endomorphism of $N$ because $\mathcal{I}$ satisfies the commutativity condition of loops. Since $N_{\gamma_0}$ is nilpotent due to $\mathcal{I}$ is admissible, $f$ is linearly independent of $1_N$ which implies $N$ is not a brick, contradict to the assumption. Therefore, there is a one-to-one correspondence between the isomorphism classes \{{\rm{brick}}$\ A$-{\rm{modules}}\} and \{{\rm{brick}}$ \ B$-{\rm{modules}}\}.

Furthermore, by viewing brick $B$-modules $M,N$ as brick $A$-modules, every $B$-module morphism from $M$ to $N$ can be seen as an $A$-module morphism. It is obvious that an $A$-module morphism $f:M\rightarrow N$ is well-defined as an $B$-module morphism. Thus, we have a natural isomorphism\[
\Hom_A(M,N)\cong\Hom_{B}(M,N).
\]\end{proof}

Proposition \ref{xxcor2.2} shows that if we want to calculate the Frobenius-Perron dimension of a bound quiver satisfying the commutativity condition, we need only consider the brick sets after removing loops. Therefore, the problems is how the extension spaces change when we remove the loops. {With pleasure,  we get the following important observation.}

\begin{theorem}
\label{xxthm2.3}
Keep the notation as in Proposition \ref{xxcor2.2}. Then the following hold.

$(1)$ If $M,N$ are two $B$-modules, and $\Hom_B(M,N)=0$, then we have\[
\Ext_A^1(M,N)\cong\Ext_B^1(M,N).
\]

$(2)$ If $M$ is a brick $B$-module and $M$ is not simple, then we have\[
\Ext_A^1(M,M)\cong\Ext_B^1(M,M).
\]
\end{theorem}

\begin{proof}
$(1)$ Since  $B$-mod is a full subcategory of $A$-mod, each element in $\Ext_B^1(M,N)$ can be viewed as an element in $\Ext_A^1(M,N)$. If $\Ext_A^1(M,N)\ncong\Ext_B^1(M,N)$, then there exists an element $\eta$ belongs to $\Ext_A^1(M,N)$ but not in $\Ext_B^1(M,N)$.
Notice that element in $\Ext_A^1(M,N)$ can be represented by a short exact sequence. $\eta$ corresponding to an exact sequence \[
\eta: \ \ 0\longrightarrow N \stackrel{\varphi}{\longrightarrow} L\stackrel{\psi}{\longrightarrow} M\longrightarrow 0
\]where $L$ is in $A$-mod but not in  $B$-mod. Which means that there exists a loop $\gamma$ that $L_{\gamma}$ is nonzero linear map. Assume that $\gamma$ is located in the vertex $P$, the non-loop arrows with the target $P$ are $\beta_i(i=1,\cdots,n)$, the non-loop arrows with the source $P$ are $\alpha_i(i=1,\cdots,m)$ and the loops located in $P$ besides $\gamma$ are $\gamma_i(i=1,\cdots,l)$. We emphasize that the above assumption including the cases there is no non-loop arrow with the target $P$, there is no non-loop arrow with the source $P$ or there is no loop located in $P$ besides $\gamma$. So the local part at $P$ is as follows 
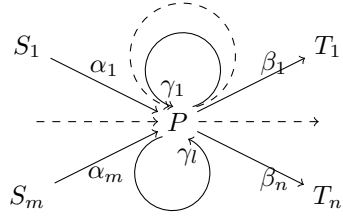
\[
\begin{tikzpicture}
\node (1) at (0,0) {$P$};
\node (2) at (-2,1) {$S_1$};
\node (3) at (-2,0){};
\node (4) at (-2,-1){$S_m$};
\node (5) at (2,1){$T_1$};
\node (6) at (2,0) {};
\node (7) at (2,-1) {$T_n$};
\draw[->] (2) --node[above ]{$\alpha_1$} (1);
\draw[->][dashed] (3) --node{} (1);
\draw[->] (4) --node[below]{$\alpha_m$} (1);
\draw[<-] (5) --node[above right]{$\beta_1$} (1);
\draw[<-][dashed] (6) --node{} (1);
\draw[<-] (7) --node[below right]{$\beta_n$} (1);
\draw [->] (0.2,0.2) arc (-75:255:0.5)node[above]{$\gamma_1$};
\draw [->][dashed] (0.25,0.2) arc (-75:255:0.7);
\draw [->] (-0.2,-0.2) arc (105:425:0.5)node[below]{$\gamma_l$};
\end{tikzpicture}
\]

Therefore, for each $\beta\in\{\beta_1,\beta_2,\cdots,\beta_n\}$, we have the following exact sequence\[
\begin{tikzpicture}
\node (1) at (-3,0) {$M_P$};
\node (2) at (0,0) {$L_P$};
\node (3) at (3,0){$N_P$};
\node (4) at (-3,-2){$M_T$};
\node (5) at (0,-2){$L_T$};
\node (6) at (3,-2) {$N_T$};
\draw[->] (1) --node[above ]{$\varphi_P$} (2);
\draw[->] (2) --node[above ]{$\psi_P$} (3);
\draw[->] (4) --node[above]{$\varphi_T$} (5);
\draw[->] (5) --node[above ]{$\psi_T$} (6);
\draw[->] (1) --node[left]{$M_\beta$} (4);
\draw[->] (2) --node[left]{$L_\beta$} (5);
\draw[->] (3) --node[left]{$N_\beta$} (6);
\draw [->] (-2.8,0.2) arc (-75:255:0.5)node[above]{$M_\gamma$};
\draw [->] (0.2,0.2) arc (-75:255:0.5)node[above]{$L_\gamma$};
\draw [->] (3.2,0.2) arc (-75:255:0.5)node[above]{$N_\gamma$};
\end{tikzpicture}
\]where $T$ is the target of $\beta$ and $M_\gamma,N_\gamma=0$.
Since as vector spaces $L_P\cong M_P\oplus N_P$ and $L_T\cong M_T\oplus N_T$. We set \[
\varphi_P=\begin{pmatrix}
   	1\\0
   	\end{pmatrix},\psi_P=(0,1),
\]
\[
\varphi_T=\begin{pmatrix}
   	1\\0
   	\end{pmatrix},\psi_T=(0,1).
\] Suppose that \[
L_\gamma=\begin{pmatrix}
   	f_{11}&f_{12} \\
   f_{21}&f_{22}
   	\end{pmatrix},
L_\beta=\begin{pmatrix}
   	g_{11}&g_{12} \\
   g_{21}&g_{22}
   	\end{pmatrix}
\]Since $L_\gamma\varphi_P=\varphi_PM_\gamma=0$, we get\[
\begin{pmatrix}
   	f_{11}&f_{12} \\
   f_{21}&f_{22}
   	\end{pmatrix}\begin{pmatrix}
   	1\\0
   	\end{pmatrix}=\begin{pmatrix}
   	f_{11}\\f_{21}
   	\end{pmatrix}=0
\]Similarly, we have $f_{22}=0$. So $L_\gamma=\begin{pmatrix}
   	0&f_{12} \\
   0&0
   	\end{pmatrix}$. Since $L_\beta L_\gamma=0$, we get\[
   \begin{pmatrix}
   	g_{11}&g_{12} \\
   g_{21}&g_{22}
   	\end{pmatrix}\begin{pmatrix}
   	0&f_{12} \\
   0&0
   	\end{pmatrix}=\begin{pmatrix}
   	0&g_{11}f_{12} \\
   0&g_{21}f_{12}
   	\end{pmatrix}=0
   \]Also by $L_\beta \varphi_P=\varphi_T M_\beta$, we get\[
   \begin{pmatrix}
   	g_{11}&g_{12} \\
   g_{21}&g_{22}
   	\end{pmatrix}\begin{pmatrix}
   	1\\0
   	\end{pmatrix}=\begin{pmatrix}
   	1\\0
   	\end{pmatrix}M_\beta
   \]Thus $g_{11}=M_\beta$ and $g_{21}=0$. It follows that $M_\beta f_{12}=0$.

   Dually, for each $\alpha\in\{\alpha_1,\alpha_2,\cdots,\alpha_m\}$, we have the following exact sequence\[
\begin{tikzpicture}
\node (1) at (-3,0) {$M_S$};
\node (2) at (0,0) {$L_S$};
\node (3) at (3,0){$N_S$};
\node (4) at (-3,-2){$M_P$};
\node (5) at (0,-2){$L_P$};
\node (6) at (3,-2) {$N_P$};
\draw[->] (1) --node[above ]{$\varphi_S$} (2);
\draw[->] (2) --node[above ]{$\psi_S$} (3);
\draw[->] (4) --node[above]{$\varphi_P$} (5);
\draw[->] (5) --node[above ]{$\psi_P$} (6);
\draw[->] (1) --node[left]{$M_\alpha$} (4);
\draw[->] (2) --node[left]{$L_\alpha$} (5);
\draw[->] (3) --node[left]{$N_\alpha$} (6);
\draw [->] (-0.2,-2.2) arc (105:425:0.5)node[below]{$L_\gamma$};
\draw [->] (-3.2,-2.2) arc (105:425:0.5)node[below]{$M_\gamma$};
\draw [->] (2.8,-2.2) arc (105:425:0.5)node[below]{$N_\gamma$};
\end{tikzpicture}
\]where $S$ is the source of $\alpha$. By similar argument as above, we have $f_{12}N_\alpha=0$.

Now we can construct a nonzero homomorphism $\theta:N\rightarrow M$ such that $\theta_P=f_{12}$ and $\theta_{P'}=0$ for each $P'\in Q_{0}\backslash \{P\}$. Which contradicts to the condition $\Hom_B(M,N)=0$. Therefore $\Ext_A^1(M,N)\cong\Ext_B^1(M,N)$.

$(2)$ If $\Ext_A^1(M,M)\ncong\Ext_B^1(M,M)$, by similar argument as $(1)$, we can construct a nonzero homomorphism $\theta$ from $M$ to $M$. Since $M$ is not a simple module, there exists at least two vertices $P_1,P_2$ such that $M_{P_1},M_{P_2}\ne0$. But there is only one vertex $P_0$ satisfies $\theta_{P_0}\ne0$. So $\theta$ is not an isomorphism. Which means $\theta$ and $1_{M}$ are linearly independent. It follows that $\dim_{k}\Hom_B(M,M)\ge2$. Contradicts to the condition that $M$ is a brick. Thus, $\Ext_A^1(M,M)\cong\Ext_B^1(M,M)$.
\end{proof}

\begin{remark}
For a simple $B$-module $S_P$ with vertex $P$ in $Q_0$, by \cite[Ch.\uppercase\expandafter{\romannumeral3}, Lemma 2.12] {ASS}, the value of $\mathrm{dim}\Ext^1_A(S_P,S_P)$ is equal to the number of loops at $P$.
\end{remark}

\section{Loop-extended algebras of representation-directed algebras}

{In this section, we study the Frobenius-Perron dimension of loop-extended algebras of representation-directed algebras. We recall the definition of representation-directed algebras at first.} Let $A$ be an algebra. Recall that a {\it path} in $A$-mod is a sequence\[
M_0\xrightarrow{f_1}M_1\xrightarrow{f_2}M_2\rightarrow\cdots\rightarrow M_{t-1}\xrightarrow{f_t}M_t,
\] where $t\ge1$, $M_0,M_1,\cdots,M_t$ are indecomposable $A$-modules and $f_1,\cdots,f_t$ are non-zero non-isomorphisms homomorphisms. A path in $A$-mod is called a {\it cycle} if its source module $M_0$ is isomorphic to its target module $M_t$. An indecomposable $A$-module that lies on no cycle in $A$-mod is called a {\it directing module}. An algebra is called {\it representation-directed} if every indecomposable $A$-module is directing. {We get the Frobenius-Perron dimension of loop-extended algebras of these algebras as follows. }

\begin{theorem}
\label{xxthm3.2}
Let $A=\Bbbk Q/I$ be a bound quiver algebra for some finite quiver $Q$, where $I$ is an admissible ideal satisfying the commutativity condition of loops.
Assume $B$ is the loop-reduced algebra of $A$. If $B$ is representation-directed, then we have\[
\fpdim{(A{\text -{\rm mod}})}=\max_{P\in Q_0}N_P
\]where $N_P$ means the number of loops with the source $P$.
\end{theorem}

\begin{proof}
First we prove that for every brick set $\phi=\{M_i\}_{i=1}^n$ of ${\rm mod}(B)$, the adjacency matrix is a strictly upper triangular matrix, that is to say, there exists a permutation $\sigma$ of $\{1,\cdots,n\}$ such that for $i\le j$,
\begin{equation}
\label{E1.1.1}
\Hom_B(M_i,\tau M_j)=0.
\end{equation}

In fact, using induction on $n$. If $n=1$, there is only one element $M_1$ in $\phi$. If $\Hom_B(M_1,\tau M_1)\ne 0$, then there exists a non-zero non-isomorphism homomorphism $f:M_1\rightarrow \tau M_1$. Note that Auslander-Reiten series $\tau M_1\rightarrow \oplus_{i=1}^k N_i\rightarrow M_1$ gives a path $\tau M_1\rightarrow N_1\rightarrow M_1$, so we get a cycle $M_1\rightarrow\tau M_1\rightarrow N_1\rightarrow M_1$ which contradicts the assumption $B$ is a representation-directed algebra.

Assume (\ref{E1.1.1}) holds for $n=k-1$. When $n=k$, then exists $i_0$, such that $\Hom_B(M_{i_0},\tau M_j)=0$ for $j=1,2,\cdots, n$. Otherwise, for each $i$, there exists $j$, such that $\Hom_B(M_i,\tau M_j)\ne0$. We can get a path\[
M_{i_1}\rightarrow \tau M_{i_2}\rightarrow N_{i_2}\rightarrow M_{i_2}\rightarrow \tau M_{i_3}\rightarrow \cdots
\]Since the brick set is finite, there exist $i_{s}, i_{t}$ such that $i_s=i_t$. And then we get a cycle which contradicts the assumption $B$ is a representation-directed algebra. Define $\sigma_1$ permuting $i_0$ with 1 and fixing the other numbers. By induction hypothesis, we can define the permutation $\sigma_2$ such that $\Hom_A(M_{\sigma_2\sigma_1(i)},\tau M_{\sigma_2\sigma_1(j)})=0$ for $1<\sigma_1(i)\le\sigma_1(j)$. Let $\sigma=\sigma_2 \sigma_1$, then we get what we need.

Now let us turn to $A$-mod. By Corollary \ref{xxcor2.2}, \[
\{\text{brick}\ A\text{-modules}\} \leftrightarrow\{\text{brick}\ B\text{-modules}\}.\] And for a brick set $\phi$, by Theorem \ref{xxthm2.3}$(1)$, we know the adjacency matrix $C(\phi_{A})$ of $\phi$ in $A$-mod is the same as the adjacency matrix $C(\phi_{B})$ of $\phi$ in $B$-mod not considering the diagonal elements. So $C(\phi_{A})$ is a upper triangular matrix. Assume $M\in\phi$, by Theorem \ref{xxthm2.3}$(2)$, we have \[
 \Ext_A(M,M)=0, \ {\rm{if}}\ M \rm{\ is \ not \ simple},
 \]and that\[
 \Ext_A(M,M)=N_P, \ {\rm{for}} \ M=S_P,
 \]where $S_P$ means the simple module at the vertex $P$, $N_P$ means the number of loops at $P$. Thus the elements in the diagonal of $C(\phi_{A})$ is nonzero means the corresponding module is simple and there are several loops at the corresponding vertex.

 Therefore the spectral radius of the adjacency matrix of a brick set in $A$-mod is the maximum of the number of loops at vertices, that is, \[
 \fpdim{(A{\text -{\rm mod}})}=\max_{P\in Q_0}N_P.\]
\end{proof}

\begin{remark}
	For any ADE quiver algebra, whatever the direction of arrows we choose, it is a representation-directed algebra. Since all of modified ADE bounded quiver algebras are loop-extended algebras of representation-directed algebras, hence by Theorem \ref{xxthm3.2}, the Frobenius-Perron dimension of representation category of these algebras is equal to the maximal number of loops at each vertex, which answers the question in \cite{W}.
\end{remark}

We point out that in Theorem \ref{xxthm3.2}, the condition ``$B$ is representation-directed" is necessary. In fact, if $B$ is not representation-directed, we have the following counterexample.

\begin{example}
	Define quiver $Q'$ as follow.
	
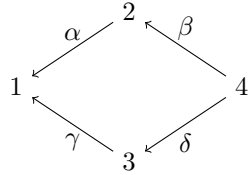
\[
	\begin{tikzpicture}
	\node (1) at (0,0) {1};
	\node (2) at (1.5,1) {2};
	\node (3) at (1.5,-1){3};
	\node (4) at (3,0){4};
	\draw[->] (2) --node[above ]{$\alpha$} (1);
	\draw[->] (3) --node[below ]{$\gamma$} (1);
	\draw[->] (4) --node[above]{$\beta$} (2);
	\draw[->] (4) --node[below ]{$\delta$} (3);
	\end{tikzpicture}
	\]
	$Q$ is the quiver formed from $Q'$ by adding $N_i$ loops to each vertex $i$. Let $A=\Bbbk Q/\langle I\cup \{\alpha\beta\}\rangle$, where $I=\{\alpha_1\alpha_2\mid \alpha_1\  \text{is a loop or }\alpha_2\  \text{is a loop}\}$. It is easy to verify that $\langle I\cup\{\alpha\beta\}\rangle$ is an admissible ideal satisfying the conditions (a),(b),(c) in Definition \ref{def3.1}. Assume $B$ is the loop-reduced algebra of $A$.
	
	We can draw the Auslander-Reiten quiver of $B$ as follows.
	
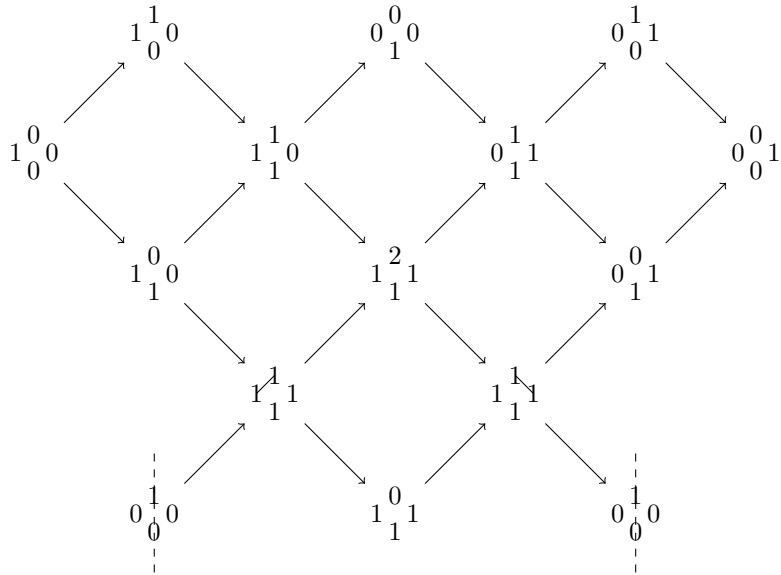
\[
	\begin{tikzpicture}
	[scale=0.8]
    \node (1) at (0+8,.3) {1};
    \node (2) at (0+8,-.3) {0};
    \node (3) at (.3+8,0){0};
    \node (4) at (-.3+8,0){0};
	\node (1) at (0,.3) {1};
	\node (2) at (0,-.3) {0};
	\node (3) at (.3,0){0};
	\node (4) at (-.3,0){0};
	\node (1) at (2,2.3) {1};
	\node (2) at (2,1.7) {1};
	\node (3) at (2.3,0+2){1};
	\node (4) at (1.7,0+2){1};
	\node (1) at (0+4,.3+4) {2};
	\node (2) at (0+4,-.3+4) {1};
	\node (3) at (.3+4,0+4){1};
	\node (4) at (-.3+4,0+4){1};
	\node (1) at (2+4,2.3+4) {1};
	\node (2) at (2+4,1.7+4) {1};
	\node (3) at (2.3+4,0+2+4){1};
	\node (4) at (1.7+4,0+2+4){0};
	\node (1) at (2+6,2.3+6) {1};
	\node (2) at (2+6,1.7+6) {0};
	\node (3) at (2.3+6,0+2+6){1};
	\node (4) at (1.7+6,0+2+6){0};
	
	\node (1) at (0,.3+4) {0};
	\node (2) at (0,-.3+4) {1};
	\node (3) at (.3,0+4){0};
	\node (4) at (-.3,0+4){1};
	\node (1) at (2,2.3+4) {1};
	\node (2) at (2,1.7+4) {1};
	\node (3) at (2.3,0+2+4){0};
	\node (4) at (1.7,0+2+4){1};
	\node (1) at (2+6-4,2.3+6) {0};
	\node (2) at (2+6-4,1.7+6) {1};
	\node (3) at (2.3+6-4,0+2+6){0};
	\node (4) at (1.7+6-4,0+2+6){0};
	
	\node (1) at (2-4,2.3+4) {0};
	\node (2) at (2-4,1.7+4) {0};
	\node (3) at (2.3-4,0+2+4){0};
	\node (4) at (1.7-4,0+2+4){1};
	\node (1) at (2+6-4-4,2.3+6) {1};
	\node (2) at (2+6-4-4,1.7+6) {0};
	\node (3) at (2.3+6-4-4,0+2+6){0};
	\node (4) at (1.7+6-4-4,0+2+6){1};
	
	\node (1) at (0+4,.3) {0};
	\node (2) at (0+4,-.3) {1};
	\node (3) at (.3+4,0){1};
	\node (4) at (-.3+4,0){1};
	\node (1) at (2+4,2.3) {1};
	\node (2) at (2+4,1.7) {1};
	\node (3) at (2.3+4,0+2){1};
	\node (4) at (1.7+4,0+2){1};
	\node (1) at (0+4+4,.3+4) {0};
	\node (2) at (0+4+4,-.3+4) {1};
	\node (3) at (.3+4+4,0+4){1};
	\node (4) at (-.3+4+4,0+4){0};
	\node (1) at (2+4+4,2.3+4) {0};
	\node (2) at (2+4+4,1.7+4) {0};
	\node (3) at (2.3+4+4,0+2+4){1};
	\node (4) at (1.7+4+4,0+2+4){0};
	
	\draw[->] (0.5,0.5) -- (1.5,1.5);
	\draw[->] (2.5,2.5) -- (3.5,3.5);
	\draw[->] (4.5,4.5) -- (5.5,5.5);
	\draw[->] (6.5,6.5) -- (7.5,7.5);
	\draw[->] (4.5,0.5) -- (5.5,1.5);
	\draw[->] (6.5,2.5) -- (7.5,3.5);
	\draw[->] (8.5,4.5) -- (9.5,5.5);
	\draw[->] (0.5,4.5) -- (1.5,5.5);
	\draw[->] (2.5,6.5) -- (3.5,7.5);
	\draw[->] (-1.5,6.5) -- (-0.5,7.5);
	
	\draw[->] (0.5,7.5) -- (1.5,6.5);
	\draw[->] (2.5,5.5) -- (3.5,4.5);
	\draw[->] (4.5,3.5) -- (5.5,2.5);
	\draw[->] (6.5,1.5) -- (7.5,0.5);
	\draw[->] (4.5,7.5) -- (5.5,6.5);
	\draw[->] (6.5,5.5) -- (7.5,4.5);
	\draw[->] (-1.5,5.5) -- (-.5,4.5);
	\draw[->] (0.5,3.5) -- (1.5,2.5);
	\draw[->] (2.5,1.5) -- (3.5,0.5);
	\draw[->] (8.5,7.5) -- (9.5,6.5);
	\draw[dashed](0,1) -- (0,-1);
	\draw[dashed](8,1) -- (8,-1);
	\draw(2,2.3) -- (1.7,2);
	\draw(6,2.3) -- (6.3,2);
	\end{tikzpicture}
	\]
	
 It is easy to find that there is a brick set consisting of the two modules at the bottom of the Auslander-Reiten quiver, whose corresponding matrix is not a upper triangular matrix. Explicitly, the corresponding matrix is $\begin{pmatrix}
	N_2 & 1 \\
	1 & 0
	\end{pmatrix}
	$, the spectral radius of which is $\dfrac{N_2+\sqrt{N_2^2+4}}{2}$.
	
 Notice that the corresponding matrices of all the other brick sets in ${A{\text -{\rm mod}}}$ are upper triangulated matrices. By the same argument as in Theorem \ref{xxthm3.2}, we have \[\fpdim{(A{\text -{\rm mod}})}=\max\{\dfrac{N_2+\sqrt{N_2^2+4}}{2},N_1,N_3,N_4\}.\]
\end{example}

\section{Some properties of tubes}

In the next two sections, we consider the loop-extended algebras of canonical algebras of type ADE and try to calculate the Frobenius-Perron dimension of the corresponding representation category. Notice that there are several tubes in the representation category of canonical algebras, we first give some properties of tubes, compared to (\cite{CG2}, Section 2.2).

Recall that in \cite{HJ}, a matrix $C$ is called {\it irreducible} if there is no permutation matrix $P$ such that \[
P^TCP=\begin{pmatrix}
    C_1&C_2\\
    0&C_3
\end{pmatrix}
\]where $C_1$ and $C_3$ are nonzero matrices.

\begin{lemma}
\label{lem5.1}
Let $\mathcal{T}$ be a tube, $\phi$ be a brick set in $\mathcal{T}$, and $C$ be the adjacency matrix of $\phi$. If $C$ is irreducible, then $C$ is a similar matrix of $$\begin{pmatrix}
    0&1\\
    &0&\ddots\\
    &&\ddots&1\\
    1&&&0
\end{pmatrix}.$$
\end{lemma}

\begin{proof}
Assume $\phi=\{M_1,M_2,\cdots,M_n\}$. For each $M_i$, there exists some $j\in\{1,\cdots,n\}$ such that $\Ext^1_{\mathcal{T}}(M_i,M_j)\ne0$ since $C$ is irreducible. If there exists $j'(j'\ne j)$ such that $\Ext^1_{\mathcal{T}}(M_i,M_{j'})\ne0$, by Serre duality, we have\[
\Hom_{\mathcal{T}}(M_{j'},\tau M_{i})\cong \Ext^1_{\mathcal{T}}(M_i,M_{j'})\ne0
\]Compared with\[
\Hom_{\mathcal{T}}(M_{j'},M_{i})=0,
\]we find that $M_{j'}$ is restricted on a coray of $\mathcal{T}$.
Similarly, we have\[
\Hom_{\mathcal{T}}(M_{j},\tau M_{i})\cong \Ext^1_{\mathcal{T}}(M_i,M_{j})\ne0
\]and\[
\Hom_{\mathcal{T}}(M_{j},M_{i})=0.
\]Then $M_j,M_{j'}$ must be on the same coray of $\mathcal{T}$, it is a contradiction to that $\phi$ is a brick set. Therefore, $M_j$ is the unique object in $\phi$ such that $\Ext^1_{\mathcal{T}}(M_i,M_j)\ne0$. We claim that $\Ext^1_{\mathcal{T}}(M_i,M_j)=\Bbbk$. In fact, if $\dim \Ext^1_{\mathcal{T}}(M_i,M_j)\ge2$, then we have\[
\dim \Hom_{\mathcal{T}}(M_j,M_j)\ge \dim \Hom_{\mathcal{T}}(M_j,\tau M_i)=\dim \Ext^1_{\mathcal{T}}(M_i,M_j)\ge2,
\] it is impossible. By a similar argument, we know that there also exists a unique object $M_k\in\phi$ such that $\Ext^1_{\mathcal{T}}(M_k,M_i)\ne0$ and we have $\Ext^1_{\mathcal{T}}(M_k,M_i)=\Bbbk$.

By the condition that $C$ is irreducible, we can find that $C$ is a similar matrix of $$\begin{pmatrix}
    0&1\\
    &0&\ddots\\
    &&\ddots&1\\
    1&&&0
\end{pmatrix}.$$
\end{proof}

\begin{corollary}
\label{cor5.2}
Let $\mathcal{T}$ be a tube, then we have\[
\fpdim(\mathcal{T})=1.
\]
\end{corollary}

\begin{proof}
Let $\phi$ be a brick set in $\mathcal{T}$, assume $\phi=\phi_1\cup\cdots\cup\phi_s$ such that the adjacency matrix of $\phi$ is\[
\begin{pmatrix}
    C_1&*&\cdots&*\\
    0&C_2&\cdots&*\\
    \vdots&\vdots&&\vdots\\
    0&0&\cdots&C_s
\end{pmatrix}
\]
where $C_i$ is the adjacency matrix of $\phi_i$ and $C_i$ is irreducible, $i=1,\cdots,s$. Then each $C_i$ has the form \[
\begin{pmatrix}
    0&1\\
    &0&\ddots\\
    &&\ddots&1\\
    1&&&0
\end{pmatrix} \text{ , } (0) \text{ or } (1).
\]So $\rho(C)=\max\{\rho(C_1),\cdots,\rho(C_s)\}\le1$, and it follows that $\fpdim(\mathcal{T})\le 1$.

On the other hand, the set consisting of all the simple objects in $\mathcal{T}$ is a brick set, and the adjacency matrix is\[
\begin{pmatrix}
    0&1\\
    &0&\ddots\\
    &&\ddots&1\\
    1&&&0
\end{pmatrix}
\]whose spectral radius is 1.
Therefore, $\fpdim(\mathcal{T})= 1.$
\end{proof}

\begin{lemma}
\label{lem5.3}
Let $\mathcal{T}$ be a tube. If there exist two different simple objects $S_1,S_2$ in $\mathcal{T}$ satisfying that $\Ext^1(S_1,S_2)=0$, then we can find an object $M$ such that\[
\Ext^1(S_1,M)=\Bbbk,\Ext^1(M,S_2)=\Bbbk.
\]and $\{S_1,S_2,M\}$ is a brick set.
\end{lemma}

\begin{proof}
We have\[
\Hom(-,\tau S_1)\cong\Ext^1(S_1,-)=\Bbbk, \Hom(\tau^{-1} S_2,-)\cong\Ext^1(-,S_2)=\Bbbk
\]which give a ray starting at $\tau^{-1} S_2$ and a coray ending at $\tau S_1$ on the tube.
Let $M$ be the intersection of the ray and the coray, by Serre duality, we get what we need.
\end{proof}

\bigskip
\section{Loop-extended algebras of canonical algebras of type ADE}

Before calculating the Frobenius-Perron dimension of this type of algebras, we need some lemmas.

\begin{lemma}
\label{lem6.1}
Let $f(x)=(x-n_1)^{r_1}(x-n_2)^{r_2}\cdots (x-n_s)^{r_s}\in \mathbb{R}[x]$, where $s\in \mathbb{Z}_{>0}$, $r_1,\cdots,r_s\in \mathbb{Z}_{>0}$ and $n_1,\cdots,n_s\in \mathbb{R}$ satisfying $0\le n_1< n_2< \cdots <n_{s-1}\le n_s-1$. Let $\{x_i\}_{i=1}^{r_1+\cdots+r_s}$ be the complete set of complex roots of $f(x)-1$.
Denote the value $\max\{|x_i|\}_{i=1}^{r_1+\cdots+r_s}$ by $\rho(f(x)-1)$. Then the follows hold.

$(1)$ $f(x)-1$ has the unique real root $x_0$ in $(n_s,n_s+1]$ and $\rho(f(x)-1)=x_0$.

$(2)$ {If $0\le m\le n_s-1$,} then $\rho((x-m)f(x)-1)<\rho(f(x)-1)$.

$(3)$ {The value $\rho((x-n_s)f(x)-1)\geq\rho(f(x)-1)$. Moreover, $\rho((x-n_s)f(x)-1)=\rho(f(x)-1)$ if and only if $s=1$.}
\end{lemma}

\begin{proof}
$(1)$ Since $f(n_s)-1<0$ {{red}and} $f(n_s+1)-1\ge0$,  $f(x)-1$ has a real root in $(n_s,n_s+1]$, {denote by $x_0$. And the derivative of $f(x)$ is} \emph{} \begin{align*}
f'(x)&=r_1(x-n_1)^{r_1-1}[(x-n_2)^{r_2}\cdots (x-n_{s-1})^{r_{s-1}}\cdot (x-n_s)^{r_s}]\\&
+\cdots+[(x-n_1)^{r_1}(x-n_2)^{r_2}\cdots (x-n_{s-1})^{r_{s-1}}]\cdot r_s(x-n_s)^{r_s-1}.
\end{align*}{It is easy to see that $f'(x)>0$ in $(n_s,n_s+1]$. So} $x_0$ is the unique real root in $(n_s,n_s+1]$.

{Moreover, if} $f(x)-1$ has a complex root $z_0$ such that $|z_0|>x_0$, we get $|z_0-n_j|>|x_0-n_j|$ {for} $j=1,\cdots,s$. Therefore, $|f(z_0)|>|f(x_0)|=1$ which is a contradiction to the condition $f(z_0)=1$. Hence $\rho(f(x)-1)=x_0$.

$(2)$ {Denote $(x-m)f(x)$ by $g(x)$. By (1), $g(x)-1$ has the unique real root $x'_0$ in $(n_s,n_s+1]$ and $\rho(g(x)-1)=x'_0$. Notice that $f(x)<g(x)$ always holds in $(n_s,n_s+1]$, $g(x)$ will meet $1$ before $f(x)$ in $(n_s,n_s+1]$, which implies $x'_0< x_{0}$. So $\rho(g(x)-1)<\rho(f(x)-1)$.}

$(3)$ {If $s=1$, then it is easy to see that $\rho((x-n_s)f(x)-1)=\rho(f(x)-1)$. So we need only consider the cases of $s>1$. Note that the unique root $x_{_{0}}$ in $(n_s,n_s+1]$ of $f(x)-1$ is also the unique real root in $(n_s,n_s+1]$ of $f(x)^{\frac{1}{n_s}}-1$.
Let} $h(x)=(x-n_1)^{r_1}(x-n_2)^{r_2}\cdots (x-n_{s-1})^{r_{s-1}}$. Then $$f(x)^{\frac{1}{n_s}}=h(x)^{\frac{1}{n_s}}\cdot (x-n_s) \ \ {\mbox{and}} \ \ [(x-n_s)f(x)]^{\frac{1}{n_s+1}}=h(x)^{\frac{1}{n_s+1}}\cdot (x-n_s).$$ Hence {$h(x)^{\frac{1}{n_s}}\cdot (x-n_s)>h(x)^{\frac{1}{n_s+1}}\cdot (x-n_s)$ in $(n_s,n_s+1]$, which implies} $h(x)^{\frac{1}{n_s}}\cdot (x-n_s)$ meets $1$ before $h(x)^{\frac{1}{n_s+1}}\cdot (x-n_s)$ in $(n_s,n_s+1]$. Therefore $\rho((x-n_s)f(x)-1)>\rho(f(x)-1)$.
\end{proof}

{Following is an example.}

\begin{example}
   {Let $f(x)=x(x-2)-1,$ $f_{1}(x)=x(x-1)(x-2)-1$ and $f_{2}(x)=x(x-2)^2-1$. Then $f_{1}(x)+1=(x-1)(f(x)+1)$ and $f_{2}(x)+1=(x-2)(f(x)+1)$. By Lemma \ref{lem6.1}(2), we have $\rho(f_{1}(x))<\rho(f(x))$. By Lemma \ref{lem6.1}(3), we have $\rho(f_{2}(x))>\rho(f(x))$. Therefore we get\[
    \rho(f_{1}(x))<\rho(f(x))<\rho(f_{2}(x)).
    \]}\emph{}
\end{example}

\bigskip

Recall a bound quiver algebra $B$ is called a canonical algebra of type ADE
if $B$ is one of the following algebras:

$(1)$ $A(n,m)=\Bbbk Q_A$ for $n\ge1,m\ge1$;

\[\begin{tikzpicture}
\node (1) at (-1,0) {$Q_A:$};
\node (1) at (0,0) {$0$};
\node (2) at (1,1) {$(1,1)$};
\node (n) at (3+2,1) {$(1,n-1)$};
\node (n+1) at (1,-1) {$(2,1)$};
\node (n+m) at (3+2,-1) {$(2,m-1)$};
\node (n+m+1) at (4+2,0) {$0'$};
\node (11) at (2.8,1) {$\cdots$};
\node (12) at (2.8,-1) {$\cdots$};
\draw[->] (n+m+1) --node[right]{$\beta_m$} (n+m);
\draw[->] (12) --node[above]{$\beta_2$} (n+1);
\draw[<-] (12) --node[above]{$\beta_{m-1}$} (n+m);
\draw[->] (n+1) --node[left]{$\beta_1$} (1);
\draw[->] (n+m+1) --node[right]{$\alpha_n$} (n);
\draw[->] (2) --node[left]{$\alpha_1$} (1);
\draw[->] (11) --node[above]{$\alpha_2$} (2);
\draw[<-] (11) --node[above]{$\alpha_{n-1}$} (n);
\end{tikzpicture}\]

$(2)$ $D_{I}(n)=\Bbbk Q_D/I$ for $n\ge4$, where $I$ is the admissible ideal of $\Bbbk Q_D$ generated by $\alpha_1\cdots\alpha_{n-2}+\beta_1\beta_2+\gamma_1\gamma_2$.

\[
\begin{tikzpicture}
\node (1) at (-1,0) {$Q_D:$};
\node (1) at (0,0) {$0$};
\node (2) at (.4,2) {$(1,1)$};
\node (n-2) at (4-.4,2) {$(1,n-3)$};
\node (n-1) at (2,1) {$(2,1)$};
\node (n) at (2,-1) {$(3,1)$};
\node (n+1) at (4,0) {$0'$};
\node (11) at (1.9,2) {$\cdots$};
\draw[->] (n+1) --node[above ]{$\beta_2$} (n-1);
\draw[->] (n+1) --node[below right]{$\gamma_2$} (n);
\draw[->] (n-1) --node[above ]{$\beta_1$} (1);
\draw[->] (n) --node[below left]{$\gamma_1$} (1);
\draw[->] (n+1) --node[above right]{$\alpha_{n-2}$} (n-2);
\draw[->] (2) --node[above left]{$\alpha_1$} (1);
\draw[->] (n-2) --node[above]{$\alpha_{n-3}$} (11);
\draw[->] (11) --node[above]{$\alpha_2$} (2);
\end{tikzpicture}\]

$(3)$ $E_{I}(n)=\Bbbk Q_E/I$ for $n=6,7,8$, where $I$ is the admissible ideal of $\Bbbk Q_E$ generated by  $\alpha_1\cdots\alpha_{n-3}+\beta_1\beta_2+\gamma_1\gamma_2\gamma_3$;

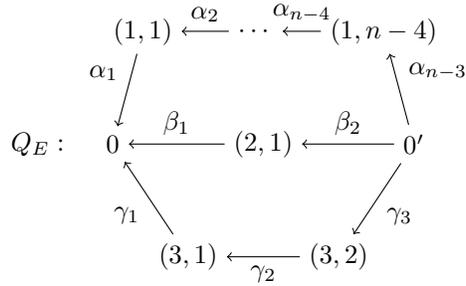
\[
\begin{tikzpicture}
\node (1) at (-1,0) {$Q_E:$};
\node (1) at (0,0) {$0$};
\node (2) at (0.4,1.5) {$(1,1)$};
\node (n-3) at (3.6,1.5) {$(1,n-4)$};
\node (n-2) at (2,0) {$(2,1)$};
\node (n-1) at (1,-1.5) {$(3,1)$};
\node (n) at (3,-1.5) {$(3,2)$};
\node (n+1) at (4,0) {$0'$};
\node (11) at (2-.1,1.5) {$\cdots$};
\draw[->] (n+1) --node[above]{$\beta_2$} (n-2);
\draw[->] (n+1) --node[below right]{$\gamma_3$} (n);
\draw[->] (n-2) --node[above ]{$\beta_1$} (1);
\draw[->] (n) --node[below]{$\gamma_2$} (n-1);
\draw[->] (n-1) --node[below left]{$\gamma_1$} (1);
\draw[->] (n+1) --node[above right]{$\alpha_{n-3}$} (n-3);
\draw[->] (2) --node[above left]{$\alpha_1$} (1);
\draw[->] (n-3) --node[above ]{$\alpha_{n-4}$} (11);
\draw[->] (11) --node[above ]{$\alpha_{2}$} (2);
\end{tikzpicture}
\]

We call $A(n,m)$ for $n\ge1,m\ge1$, $D_{I}(n)$ for $n\ge4$ and $E_{I}(n)$ for $n=6,7,8$ the canonical algebra of type A, D and E, respectively.

Let $A=\Bbbk Q/I$ be a bound quiver algebra for some finite quiver $Q$, where $I$ is an admissible ideal satisfying the commutativity condition of loops and the loop-reduced algebra $B$ of $A$ is a canonical algebra of type ADE.

Denote the number of the loops at the sink vertex $0$ by $n_0$; denote the number of the loops at the source vertex $0'$ by $n_{0'}$; denote the number of the loops at the other vertexes $(i,j)$ by $n_{ij}$. Let $S=\{n_{ij}\}_{i,j}\bigcup\{n_0,n_{0'}\}$, $n_{max}=\max S$. {The following result shows that the Frobenius-Perron dimension of $A$ is highly depended on the location of $n_{max}$.}

\begin{theorem}
\label{thm6.3}

Let $A$ be the algebra defined as above, then we have
\[
\fpdim{(A{\text -{\rm mod}})}\in [n_{max},n_{max}+1).
\]In addition, the Frobenius-Perron dimension can be calculated precisely in the following cases.

$(1)$ If $n_{max}=\max\{n_0,n_{0'}\}$ and $n_{max}>\max\{n_{ij}\}_{i,j}$, then we have\[
\fpdim{(A{\text -{\rm mod}})}=n_{max}.
\]

$(2)$ If $n_{max}=n_{i_0j_0}$ for some $(i_0,j_0)$, and $n_{max}>\max\{n_{ij}\}_{i,j},(i,j)\ne (i_0,j_0)$, then we have\[
\fpdim{(A{\text -{\rm mod}})}=\dfrac{n_{max}+\sqrt{4+n_{max}^2}}{2}.
\]

$(3)$ If we have $\begin{cases}
n_{ij}=n_{max},(i,j)=(i_0,j_0),(i_0,j_0+1),\cdots,(i_0,j_0+s-1)\\
n_{ij}<n_{max},others
\end{cases}$,\\
then we get\[
\fpdim{(A{\text -{\rm mod}})}=\rho(x(x-n_{max})^s-1).
\]
\end{theorem}

\begin{proof}
By Proposition \ref{xxcor2.2}, we only need to consider the brick set in $B$-mod. {Since} each indecomposable object of $B$-mod is in $\mathcal{P}$, $\mathcal{R}$ or $\mathcal{I}$, where $\mathcal{P}, \mathcal{R},\mathcal{I}$ is the postprojective, regular, preinjective component of $B$-mod,  {respectively}. And there are no non-zero morphisms from $\mathcal{R}$ to $\mathcal{P}$, $\mathcal{I}$ to $\mathcal{P}$ or $\mathcal{I}$ to $\mathcal{R}$. In other words, $B$-mod=$\mathcal{P}\vee \mathcal{R}\vee\mathcal{I}$. Hence $\Ext^1_A(\mathcal{R,P})=0$ and $\Ext_A^1(\mathcal{I,R})=0$, we have\[
\fpdim{(A{\text -{\rm mod}})}=\max\{\fpdim{(\mathcal{P})},\fpdim{(\mathcal{R})},\fpdim{(\mathcal{I})}\}.\] In addition, any adjacency matrix of a brick set in $\mathcal{P}$ or $\mathcal{I}$ is a upper triangular matrix ({by a suitable order of objects in the brick set}), and the diagonal element is nonzero if and only if its corresponding module is simple and there are loops at the corresponding vertex. By lemma \ref{lem5.1}, any irreducible adjacency matrix of a brick set in $\mathcal{R}$ has the following form ({by a suitable order of objects in the brick set}) \[
\begin{pmatrix}
    d_1&1\\
    &d_2&\ddots\\
    &&\ddots&1\\
    1&&&d_N
\end{pmatrix}
\]and the diagonal element is {non-zero} if and only if its corresponding module is simple and there {exist} loops at corresponding vertex. Since $\fpdim{(\mathcal{R})}<\max\{n_{ij}\}_{i,j}+1$, we get $n_{max}\le\fpdim({{\rm mod}(A)})<\max\{\max\{n_{ij}\}_{i,j}+1,n_0+1,n_0'+1\}=n_{max}+1.$

$(1)$ If $n_{max}=n_0$, which means that $\dim \Ext^1_A(S_{0},S_{0})=n_1$, where $S_{0}$ is the simple module at the sink vertex $0$. Hence we have $\fpdim{(\mathcal{I})}=n_{max}$. On the other hand, we have $\fpdim{(\mathcal{R})}<\max\{n_{ij}\}_{i,j}+1\le n_{max},\fpdim{(\mathcal{P})}\le n_{max}$. Therefore, $\fpdim{(A{\text -{\rm mod}})}=n_{max}.$ Similarly, if $n_{max}=n_{0'}$, then $\fpdim{(A{\text -{\rm mod}})}=n_{max}.$

$(2)$ If $n_{max}=n_{i_0,j_0}$. First we have $\fpdim{(\mathcal{P})}\le n_{max}$ and $\fpdim{(\mathcal{I})}\le n_{max}$. Then we consider the brick set in $\mathcal{R}$.

According to Lemma \ref{lem5.3}, we can find a module $M$ such that $\{M,S_{i_0,j_0}\}$ is a brick set and the adjacency matrix is $$
\begin{pmatrix}
    0&1\\1&n_{max}
\end{pmatrix}
$$ whose spectral radius is $\dfrac{n_{max}+\sqrt{4+n_{max}^2}}{2}$.

For any other irreducible adjacency matrix of a brick set containing $S_{i_0,j_0}$ in $\mathcal{R}$, the set of the diagonal elements denoted by $Diag$ must contain $\{0,n_{max}\}$ and $\max(Diag-\{n_{max},0\})\le n_{max}-1$. Therefore, according to Lemma \ref{lem6.1}$(2)$, its spectral radius is less than $\dfrac{n_{max}+\sqrt{4+n_{max}^2}}{2}$. Hence we have $\fpdim{(A{\text -{\rm mod}})}=\dfrac{n_{max}+\sqrt{4+n_{max}^2}}{2}$.

$(3)$ According to Lemma \ref{lem5.3}, we can find a brick module $M$ to form a new brick set $\{M,S_{i_0,j_0},\cdots,S_{i_0,j_0+s-1}\}$ and the adjacency matrix is $$
\begin{pmatrix}
    0&1\\
    &n_{max}&\ddots\\
    &&\ddots&1\\
    1&&&n_{max}
\end{pmatrix}.
$$ By Lemma \ref{lem6.1}$(3)$, we know\[
\rho(x(x-n_{max})^s-1)>\rho(x(x-n_{max})^{s-1}-1)>\cdots>\rho(x(x-n_{max})-1).
\]It follows that $\fpdim{(A{\text -{\rm mod}})}=\rho(x(x-n_{max})^s-1)$.
\end{proof}
\bigskip

For a set $\{n_{ij}\}_{i,j}$ not satisfying the cases $(1)(2)(3)$ in Theorem \ref{thm6.3}, there is no obvious relationship of the
{Frobenius-Perron dimension}. Sometimes we need to consider $\{n_{ij}\}_{i,j}$ other than $n_{max}$, and sometimes we need not. In the case other than $(1)(2)(3)$ in Theorem \ref{thm6.3}, we have to get the Frobenius-Perron by concrete calculating. We present some examples as follows.

\begin{example}
Keep the notation as in Theorem \ref{thm6.3}. Denote the quiver of $B$ by $Q_0$.

(1) If the follow is the quiver of $B$.
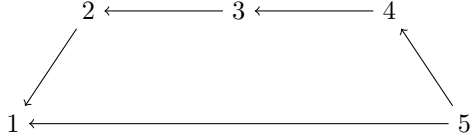
	\[
	\begin{tikzpicture}
	\node (1) at (0,0) {1};
	\node (2) at (1,1.5) {2};
	\node (3) at (3,1.5){3};
	\node (4) at (5,1.5){4};
	\node (5) at (6,0){5};
	\draw[->] (2) --node[above ]{} (1);
	\draw[->] (3) --node[below ]{} (2);
	\draw[->] (4) --node[above]{} (3);
	\draw[->] (5) --node[below ]{} (4);
	\draw[->] (5) --node[below ]{} (1);
	\end{tikzpicture}
	\]
    The corresponding numbers of loops at vertexes $1,2,3,4,5$ are $n_1=0,n_2=2,n_3=1,n_4=2,n_5=0$, then we find that\[
	\rho((x-2)^2(x-1)x-1)>\rho((x-2)x-1).
	\]So we get\[
	\fpdim{(A{\text -{\rm mod}})}=\rho((x-2)^2(x-1)x-1)
	\]In this case, the Frobenius-Perron dimension is determined not only by $n_{max}$, but also by the numbers of loops at other vertexes.

    (2) If the follow is the quiver of $B$.
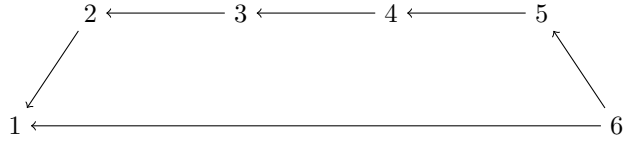
    \[
	\begin{tikzpicture}
	\node (1) at (0,0) {1};
	\node (2) at (1,1.5) {2};
	\node (3) at (3,1.5){3};
	\node (4) at (5,1.5){4};
	\node (5) at (7,1.5){5};
	\node (6) at (8,0){6};
	\draw[->] (2) --node[above ]{} (1);
	\draw[->] (3) --node[below ]{} (2);
	\draw[->] (4) --node[above]{} (3);
	\draw[->] (5) --node[below ]{} (4);
	\draw[->] (6) --node[below ]{} (5);
	\draw[->] (6) --node[below ]{} (1);
	\end{tikzpicture}
	\]
	The corresponding numbers of loops at vertexes $1,2,3,4,5,6$ are $n_1=0,n_2=3,n_3=1,n_4=1,n_5=3,n_6=0$, then we find that\[
	\rho((x-3)^2(x-1)^2x-1)<\rho((x-3)x-1).
	\]So we get\[
	\fpdim{(A{\text -{\rm mod}})}=\rho((x-3)x-1).
	\]In this case, the Frobenius-Perron dimension is just determined by $n_{max}$.
	
	(3) If the quiver of $B$ is the same as (2).
	The corresponding numbers of loops at vertexes $1,2,3,4,5,6$ are $n_1=0,n_2=6,n_3=4,n_4=5,n_5=6,n_6=0$, then we find that\[
	\rho((x-6)^2(x-5)(x-4)x-1)>\rho((x-6)x-1).
	\]So we get\[
	\fpdim{(A{\text -{\rm mod}})}=\rho((x-6)^2(x-5)(x-4)x-1).
	\]In this case, note that the ratio of the numbers of loops at vertexes $3,4$ to $n_{max}=n_2=n_5$ are bigger than $(2)$.
	Although the quiver of $B$ is the same as $(2)$, the Frobenius-Perron dimension is determined not only by $n_{max}$, but also by the numbers of loops at other vertexes.
\end{example}

\bigskip
\section{Polynomial algebras}

In this section, we focus on the polynomial algebras $\Bbbk[x_1,x_2,\cdots,x_r]$ {for $r \in\mathbb{Z}_{\geq1}$} which are infinite dimensional algebras. We will calculate the Frobenius-Perron dimension of the representation categories of these algebras.

As is known that Auslander-Reiten quiver of the finite dimensional representation category of $\Bbbk[x]$ consists of tubes, so the Frobenius-Perron dimension is 1 by Corollary \ref{cor5.2}. Following, we consider the representation category of $\Bbbk[x_1,x_2]$, which is denoted by $\mathcal{REP}$ in this section. It is obvious that a representation in $\mathcal{REP}$ can be written as
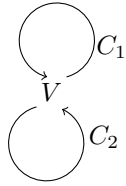
\[
\begin{tikzpicture}
\node (1) at (0,0) {$V$};
\node (1) at (.8,.6) {$C_1$};
\node (1) at (0.7,-.6) {$C_2$};
\draw [->] (0.2,0.2) arc (-75:255:0.5)node[above]{};
\draw [->] (-0.2,-0.2) arc (105:425:0.5)node[below]{};
\end{tikzpicture}
\]where $V$ is a $\Bbbk$-linear space of dimension $n$, $C_1,C_2$ are $n\times n$ matrices satisfying $C_1C_2=C_2C_1$. For simplicity, we denote the representation by $(V,C_1,C_2)$.

\begin{lemma}
\label{lem7.1}
A representation in $\mathcal{REP}$ is a brick if and only if it is of one dimension. In addition, there is no {non-zero morphism} between two different bricks, so arbitrary {many bricks which are different from each other} constitute a brick set.
\end{lemma}

\begin{proof}
Obviously, a representation of one dimension is a brick. Conversely, let $M=(V,C_1,C_2)$ be a brick in $\mathcal{REP}$, {$D$} be an endmorphism of $M$, i.e.
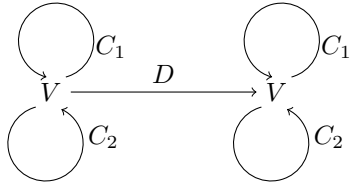
\[
\begin{tikzpicture}
\node (2) at (0,0) {$V$};
\node (1) at (.8,.6) {$C_1$};
\node (1) at (0.7,-.6) {$C_2$};
\draw [->] (0.2,0.2) arc (-75:255:0.5)node[above]{};
\draw [->] (-0.2,-0.2) arc (105:425:0.5)node[below]{};
\node (3) at (1+2,0) {$V$};
\node (1) at (1+2.8,.6) {$C_1$};
\node (1) at (3.7,-.6) {$C_2$};
\draw [->] (3.2,0.2) arc (-75:255:0.5)node[above]{};
\draw [->] (1+2-0.2,-0.2) arc (105:425:0.5)node[below]{};
\draw[->] (2) --node[above ]{{$D$}} (3);
\end{tikzpicture}
\]Then we can choose {$D=\mathrm{Id}, D=C_1$ or $D=C_2.$} Since $M$ is a brick, it follows that $C_1=\lambda \mathrm{Id},C_2=\mu \mathrm{Id}$ for $\lambda,\mu\in\Bbbk$. In this case, any decomposition $V=V_1\oplus V_2$ implies a decomposition $M=M_1 \oplus M_2$. Therefore, $V$ must be of one dimension.

For two bricks $(\Bbbk,\lambda_1,\mu_1),(\Bbbk,\lambda_2,\mu_2)\in \mathcal{REP}$ {where $\lambda_1,\lambda_2,\mu_1,\mu_2\in\Bbbk$}, assume there is {a non-zero morphism $\nu$} as follows
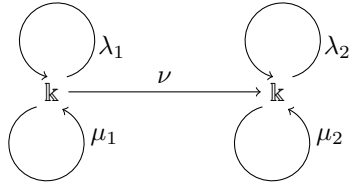
\[
\begin{tikzpicture}
\node (2) at (0,0) {$\Bbbk$};
\node (1) at (.8,.6) {$\lambda_1$};
\node (1) at (0.7,-.6) {$\mu_1$};
\draw [->] (0.2,0.2) arc (-75:255:0.5)node[above]{};
\draw [->] (-0.2,-0.2) arc (105:425:0.5)node[below]{};
\node (3) at (1+2,0) {$\Bbbk$};
\node (1) at (1+2.8,.6) {$\lambda_2$};
\node (1) at (3.7,-.6) {$\mu_2$};
\draw [->] (3.2,0.2) arc (-75:255:0.5)node[above]{};
\draw [->] (1+2-0.2,-0.2) arc (105:425:0.5)node[below]{};
\draw[->] (2) --node[above ]{$\nu$} (3);
\end{tikzpicture}
\]  {Then by the condition, we get $\lambda_1=\lambda_2$ and $\mu_1=\mu_2$.}
\end{proof}

Now we will try to calculate the extension space between bricks.

\begin{lemma}
\label{lem7.2}
For two representations $(V,C_1,C_2),(W,B_1,B_2)\in \mathcal{REP}$ and $\lambda,\mu\in\Bbbk$, there is an isomorphism of $\Bbbk$-linear spaces\[
\Ext^1_\mathcal{REP}((V,C_1,C_2),(W,B_1,B_2))\cong \Ext^1_\mathcal{REP}((V,C'_1,C'_2),(W,B'_1,B'_2))
\]where {$C'_1=C_1+\lambda \mathrm{Id},C'_2=C_2+\mu \mathrm{Id},
B'_1=B_1+\lambda \mathrm{Id} $ and $B'_2=B_2+\mu \mathrm{Id}$.}
\end{lemma}

\begin{proof}
{Let {$\xi$} be an element in $\Ext^1_\mathcal{REP}((V,C_1,C_2),(W,B_1,B_2))$. Then $\xi$ is} an short exact sequence {represented}  as follow
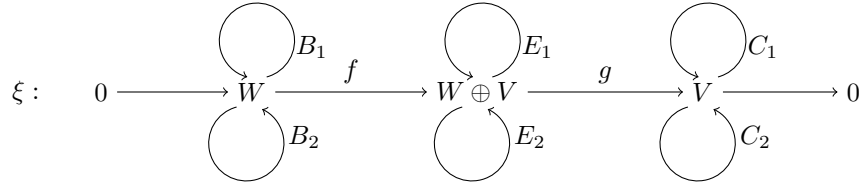
\[
\begin{tikzpicture}
\node (-2) at (-3,0) {$\xi:$};
\node (-2) at (-2,0) {0};
\node (-1) at (8,0) {0};
\node (2) at (0,0) {$W$};
\draw[->] (-2) --node[above ]{} (2);
\node (1) at (.8,.6) {$B_1$};
\node (1) at (0.7,-.6) {$B_2$};
\draw [->] (0.2,0.2) arc (-75:255:0.5)node[above]{};
\draw [->] (-0.2,-0.2) arc (105:425:0.5)node[below]{};
\node (3) at (1+2,0) {$W\oplus V$};
\node (1) at (1+2.8,.6) {$E_1$};
\node (1) at (3.7,-.6) {$E_2$};
\draw [->] (3.2,0.2) arc (-75:255:0.5)node[above]{};
\draw [->] (1+2-0.2,-0.2) arc (105:425:0.5)node[below]{};
\draw[->] (2) --node[above ]{$f$} (3);
\node (4) at (3+1+2,0) {$V$};
\draw[->] (4) --node[above ]{} (-1);
\node (1) at (1+3+2.8,.6) {$C_1$};
\node (1) at (6.7,-.6) {$C_2$};
\draw [->] (6.2,0.2) arc (-75:255:0.5)node[above]{};
\draw [->] (3+1+2-0.2,-0.2) arc (105:425:0.5)node[below]{};
\draw[->] (3) --node[above ]{$g$} (4);
\end{tikzpicture}
\]satisfying  $$B_1B_2=B_2B_1,C_1C_2=C_2C_1,E_1E_2=E_2E_1,$$
$$f B_1=E_1f,f B_2=E_2f,$$
$$gE_1=C_1g,gE_2=C_2g.$$ {It is easy to check that the following  short exact sequence $\zeta$} is an element in $\Ext^1_\mathcal{REP}((V,C'_1,C'_2),(W,B'_1,B'_2))$ 
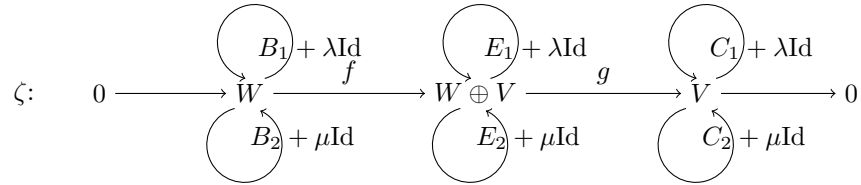
\[
\begin{tikzpicture}
\node (-2) at (-3,0) {$\zeta$:};
\node (-2) at (-2,0) {0};
\node (-1) at (8,0) {0};
\node (2) at (0,0) {$W$};
\draw[->] (-2) --node[above ]{} (2);
\node (1) at (.8,.6) {$B_1+\lambda \mathrm{Id}$};
\node (1) at (0.7,-.6) {$B_2+\mu \mathrm{Id}$};
\draw [->] (0.2,0.2) arc (-75:255:0.5)node[above]{};
\draw [->] (-0.2,-0.2) arc (105:425:0.5)node[below]{};
\node (3) at (1+2,0) {$W\oplus V$};
\node (1) at (1+2.8,.6) {$E_1+\lambda \mathrm{Id}$};
\node (1) at (3.7,-.6) {$E_2+\mu \mathrm{Id}$};
\draw [->] (3.2,0.2) arc (-75:255:0.5)node[above]{};
\draw [->] (1+2-0.2,-0.2) arc (105:425:0.5)node[below]{};
\draw[->] (2) --node[above ]{$f$} (3);
\node (4) at (3+1+2,0) {$V$};
\draw[->] (4) --node[above ]{} (-1);
\node (1) at (1+3+2.8,.6) {$C_1+\lambda \mathrm{Id}$};
\node (1) at (6.7,-.6) {$C_2+\mu \mathrm{Id}$};
\draw [->] (6.2,0.2) arc (-75:255:0.5)node[above]{};
\draw [->] (3+1+2-0.2,-0.2) arc (105:425:0.5)node[below]{};
\draw[->] (3) --node[above ]{$g$} (4);
\end{tikzpicture}
\]
Hence {we can construct a corresponding} \begin{align*}
\psi:\Ext^1_\mathcal{REP}((V,C_1,C_2),(W,B_1,B_2))&\rightarrow \Ext^1_\mathcal{REP}((V,C'_1,C'_2),(W,B'_1,B'_2)),\\& {\xi\mapsto\zeta.}
\end{align*}We will prove that $\psi$ is an isomorphism of $\Bbbk$-vector spaces.

{Firstly, we claim that $\psi$ is well-defined.  In} $\Ext^1_\mathcal{REP}((V,C_1,C_2),(W,B_1,B_2))$, if $\xi_{1}=\xi_{2}$, it means that there exists a {isomorphism $h$ such that the following diagram is commutative} \[
\begin{tikzpicture}
\node (-2) at (-3,0) {$\xi_{1}:$};
\node (-2) at (-2,0) {0};
\node (-1) at (8,0) {0};
\node (2) at (0,0) {$W$};
\draw[->] (-2) --node[above ]{} (2);
\node (1) at (.8,.6) {$B_1$};
\node (1) at (0.7,-.6) {$B_2$};
\draw [->] (0.2,0.2) arc (-75:255:0.5)node[above]{};
\draw [->] (-0.2,-0.2) arc (105:425:0.5)node[below]{};
\node (3) at (1+2,0) {$W\oplus V$};
\node (1) at (1+2.8,.6) {$E_1$};
\node (1) at (3.7,-.6) {$E_2$};
\draw [->] (3.2,0.2) arc (-75:255:0.5)node[above]{};
\draw [->] (1+2-0.2,-0.2) arc (105:425:0.5)node[below]{};
\draw[->] (2) --node[above ]{$f$} (3);
\node (4) at (3+1+2,0) {$V$};
\draw[->] (4) --node[above ]{} (-1);
\node (1) at (1+3+2.8,.6) {$C_1$};
\node (1) at (6.7,-.6) {$C_2$};
\draw [->] (6.2,0.2) arc (-75:255:0.5)node[above]{};
\draw [->] (3+1+2-0.2,-0.2) arc (105:425:0.5)node[below]{};
\draw[->] (3) --node[above ]{$g$} (4);
\draw[-] (0,-.3)--(0,-2.7);
\draw[-] (0.1,-.3)--(0.1,-2.7);
\draw[-] (0+6,-.3)--(0+6,-2.7);
\draw[-] (0.1+6,-.3)--(0.1+6,-2.7);
\draw[->] (0+3,-.3)--node[right]{$h$}(0+3,-2.7);
\node (-2) at (-3,0-3) {$\xi_{2}:$};
\node (-2) at (-2,0-3) {0};
\node (-1) at (8,0-3) {0};
\node (2) at (0,0-3) {$W$};
\draw[->] (-2) --node[above ]{} (2);
\node (1) at (.8,.6-3) {$B_1$};
\node (1) at (0.7,-.6-3) {$B_2$};
\draw [->] (0.2,0.2-3) arc (-75:255:0.5)node[above]{};
\draw [->] (-0.2,-0.2-3) arc (105:425:0.5)node[below]{};
\node (3) at (1+2,0-3) {$W\oplus V$};
\node (1) at (1+2.8,.6-3) {${\hat E_1}$};
\node (1) at (3.7,-.6-3) {${\hat E_2}$};
\draw [->] (3.2,0.2-3) arc (-75:255:0.5)node[above]{};
\draw [->] (1+2-0.2,-0.2-3) arc (105:425:0.5)node[below]{};
\draw[->] (2) --node[above ]{${\hat f}$} (3);
\node (4) at (3+1+2,0-3) {$V$};
\draw[->] (4) --node[above ]{} (-1);
\node (1) at (1+3+2.8,.6-3) {$C_1$};
\node (1) at (6.7,-.6-3) {$C_2$};
\draw [->] (6.2,0.2-3) arc (-75:255:0.5)node[above]{};
\draw [->] (3+1+2-0.2,-0.2-3) arc (105:425:0.5)node[below]{};
\draw[->] (3) --node[above ]{${\hat g}$} (4);
\end{tikzpicture}
\]{From} the commutativity of the diagram above, we have\[
\begin{tikzpicture}
\node (-2) at (-3,0) {$\psi(\xi_{1}$):
};
\node (-2) at (-2,0) {0};
\node (-1) at (8,0) {0};
\node (2) at (0,0) {$W$};
\draw[->] (-2) --node[above ]{} (2);
\node (1) at (.8,.6) {$B_1+\lambda \mathrm{Id}$};
\node (1) at (0.7,-.6) {$B_2+\mu \mathrm{Id}$};
\draw [->] (0.2,0.2) arc (-75:255:0.5)node[above]{};
\draw [->] (-0.2,-0.2) arc (105:425:0.5)node[below]{};
\node (3) at (1+2,0) {$W\oplus V$};
\node (1) at (1+2.8,.6) {$E_1+\lambda \mathrm{Id}$};
\node (1) at (3.7,-.6) {$E_2+\mu \mathrm{Id}$};
\draw [->] (3.2,0.2) arc (-75:255:0.5)node[above]{};
\draw [->] (1+2-0.2,-0.2) arc (105:425:0.5)node[below]{};
\draw[->] (2) --node[above ]{$f$} (3);
\node (4) at (3+1+2,0) {$V$};
\draw[->] (4) --node[above ]{} (-1);
\node (1) at (1+3+2.8,.6) {$C_1+\lambda \mathrm{Id}$};
\node (1) at (6.7,-.6) {$C_2+\mu \mathrm{Id}$};
\draw [->] (6.2,0.2) arc (-75:255:0.5)node[above]{};
\draw [->] (3+1+2-0.2,-0.2) arc (105:425:0.5)node[below]{};
\draw[->] (3) --node[above ]{$g$} (4);
\draw[-] (0,-.3)--(0,-2.7);
\draw[-] (0.1,-.3)--(0.1,-2.7);
\draw[-] (0+6,-.3)--(0+6,-2.7);
\draw[-] (0.1+6,-.3)--(0.1+6,-2.7);
\draw[->] (0+3,-.3)--node[right]{$h$}(0+3,-2.7);
\node (-2) at (-3,0-3) {$\psi(\xi_{2}):$};
\node (-2) at (-2,0-3) {0};
\node (-1) at (8,0-3) {0};
\node (2) at (0,0-3) {$W$};
\draw[->] (-2) --node[above ]{} (2);
\node (1) at (.8,.6-3) {$B_1+\lambda \mathrm{Id}$};
\node (1) at (0.7,-.6-3) {$B_2+\mu \mathrm{Id}$};
\draw [->] (0.2,0.2-3) arc (-75:255:0.5)node[above]{};
\draw [->] (-0.2,-0.2-3) arc (105:425:0.5)node[below]{};
\node (3) at (1+2,0-3) {$W\oplus V$};
\node (1) at (1+2.8,.6-3) {${\hat E_1}+\lambda \mathrm{Id}$};
\node (1) at (3.7,-.6-3) {${\hat E_2}+\mu \mathrm{Id}$};
\draw [->] (3.2,0.2-3) arc (-75:255:0.5)node[above]{};
\draw [->] (1+2-0.2,-0.2-3) arc (105:425:0.5)node[below]{};
\draw[->] (2) --node[above ]{${\hat f}$} (3);
\node (4) at (3+1+2,0-3) {$V$};
\draw[->] (4) --node[above ]{} (-1);
\node (1) at (1+3+2.8,.6-3) {$C_1+\lambda \mathrm{Id}$};
\node (1) at (6.7,-.6-3) {$C_2+\mu \mathrm{Id}$};
\draw [->] (6.2,0.2-3) arc (-75:255:0.5)node[above]{};
\draw [->] (3+1+2-0.2,-0.2-3) arc (105:425:0.5)node[below]{};
\draw[->] (3) --node[above ]{${\hat g}$} (4);
\end{tikzpicture}
\]
{is commutative. It follows that $\psi(\xi_{1})=\psi(\xi_{2})$.}

{Secondly,} $\psi$ keeps the scalar-multiplication. In fact, for $0\ne a\in\Bbbk$, {we have $a\cdot\psi(\xi)=\psi(a\cdot\xi$)} is the short exact sequence as follow\[
\begin{tikzpicture}
\node (-2) at (-2,0) {0};
\node (-1) at (8,0) {0};
\node (2) at (0,0) {$W$};
\draw[->] (-2) --node[above ]{} (2);
\node (1) at (.8,.6) {$B_1+\lambda \mathrm{Id}$};
\node (1) at (0.7,-.6) {$B_2+\mu \mathrm{Id}$};
\draw [->] (0.2,0.2) arc (-75:255:0.5)node[above]{};
\draw [->] (-0.2,-0.2) arc (105:425:0.5)node[below]{};
\node (3) at (1+2,0) {$W\oplus V$};
\node (1) at (1+2.8,.6) {$E_1+\lambda \mathrm{Id}$};
\node (1) at (3.7,-.6) {$E_2+\mu \mathrm{Id}$};
\draw [->] (3.2,0.2) arc (-75:255:0.5)node[above]{};
\draw [->] (1+2-0.2,-0.2) arc (105:425:0.5)node[below]{};
\draw[->] (2) --node[above ]{$f$} (3);
\node (4) at (3+1+2,0) {$V$};
\draw[->] (4) --node[above ]{} (-1);
\node (1) at (1+3+2.8,.6) {$C_1+\lambda \mathrm{Id}$};
\node (1) at (6.7,-.6) {$C_2+\mu \mathrm{Id}$};
\draw [->] (6.2,0.2) arc (-75:255:0.5)node[above]{};
\draw [->] (3+1+2-0.2,-0.2) arc (105:425:0.5)node[below]{};
\draw[->] (3) --node[above right]{$\dfrac{1}{a}g$} (4);
\end{tikzpicture}
\]

{Furthermore, $\psi$ keeps the addition. In fact, assume $\xi_{1}, \xi_{2}$ are two elements in $\Ext^1_\mathcal{REP}((V,C_1,C_2),(W,B_1,B_2))$. Then the sum of $\xi_{1}, \xi_{2}$} can be expressed to be the Baer sum\[
\xi_{1}+\xi_{2}=(1,1)(\xi_{1}\oplus \xi_{2})\begin{pmatrix}
1\\1
\end{pmatrix}
\]which is the pull-back of the push-out (or the push-out of the pull-back). In other words, we have the {commutative diagram} as follow\[
\begin{tikzpicture}
\node (-2) at (.1-3.5,0) {$\xi_{1}\oplus \xi_{_{2}}:$};
\node (-2) at (-2,0) {0};
\node (-1) at (8,0) {0};
\node (2) at (0,0) {$W\oplus W$};
\draw[->] (-2) --node[above ]{} (2);
\node (1) at (.8,.6) {diag($B_1,{B_1}$)};
\node (1) at (0.7,-.8) {diag($B_2,{B_2}$)};
\draw [->] (0.2,0.2) arc (-75:255:0.5)node[above]{};
\draw [->] (-0.2,-0.2) arc (105:425:0.5)node[below]{};
\node (3) at (1+2,0) {$W\oplus V\oplus W\oplus V$};
\node (1) at (1+2.8,.6) {diag($E_1,{\hat E_1}$)};
\node (1) at (3.7,-.6) {diag($E_2,{\hat E_2}$)};
\draw [->] (3.2,0.2) arc (-75:255:0.5)node[above]{};
\draw [->] (1+2-0.2,-0.2) arc (105:425:0.5)node[below]{};
\draw[->] (2) --node[below]{diag($f,{\hat f}$)} (3);
\node (4) at (3+1+2,0) {$V\oplus V$};
\draw[->] (4) --node[above ]{} (-1);
\node (1) at (1+3+2.8,.6) {diag($C_1,{C_1}$)};
\node (1) at (6.7,-.6) {diag($C_2,{C_2}$)};
\draw [->] (6.2,0.2) arc (-75:255:0.5)node[above]{};
\draw [->] (3+1+2-0.2,-0.2) arc (105:425:0.5)node[below]{};
\draw[->] (3) --node[above ]{diag($g,{\hat g}$)} (4);
\draw[->] (0,-.3)--node[left]{$(1,1)$}(0,-2.7);
\draw[-] (0+6,-.3)--(0+6,-2.7);
\draw[-] (0.1+6,-.3)--(0.1+6,-2.7);
\draw[->] (0+3,-.3)--node[right]{}(0+3,-2.7);
\node (-2) at (.1-3,0.5-3) {$(1,1)(\xi_{1}\oplus \xi_{2}):$};
\node (-2) at (-2,0-3) {0};
\node (-1) at (8,0-3) {0};
\node (2) at (0,0-3) {$W$};
\draw[->] (-2) --node[above ]{} (2);
\node (1) at (.8,.6-3) {$B_1$};
\node (1) at (0.7,-.6-3) {$B_2$};
\draw [->] (0.2,0.2-3) arc (-75:255:0.5)node[above]{};
\draw [->] (-0.2,-0.2-3) arc (105:425:0.5)node[below]{};
\node (3) at (1+2,0-3) {$W\oplus V\oplus V$};
\node (1) at (1+2.8,.6-3) {};
\node (1) at (3.7,-.6-3) {};
\draw [->] (3.2,0.2-3) arc (-75:255:0.5)node[above]{};
\draw [->] (1+2-0.2,-0.2-3) arc (105:425:0.5)node[below]{};
\draw[->] (2) --node[above ]{} (3);
\node (4) at (3+1+2,0-3) {$V\oplus V$};
\draw[->] (4) --node[above ]{} (-1);
\node (1) at (1+3+2.8,.6-3) {diag($C_1,{C_1}$)};
\node (1) at (6.7,-.6-3) {diag($C_2,{C_2}$)};
\draw [->] (6.2,0.2-3) arc (-75:255:0.5)node[above]{};
\draw [->] (3+1+2-0.2,-0.2-3) arc (105:425:0.5)node[below]{};
\draw[->] (3) --node[above ]{} (4);
\draw[-] (0,-.3-3)--node[left]{}(0,-2.7-3);
\draw[-] (0.1,-.3-3)--node[left]{}(0.1,-2.7-3);
\draw[<-] (0.1+6,-.3-3)--node[right]{$(1,1)^T$}(0.1+6,-2.7-3);
\draw[<-] (0+3,-.3-3)--node[right]{}(0+3,-2.7-3);
\node (-2) at (.5-3,0.7-6) {$(1,1)(\xi_{1}\oplus \xi_{2})\begin{pmatrix}
	1\\1
	\end{pmatrix}:$};
\node (-2) at (-2,0-6) {0};
\node (-1) at (8,0-6) {0.};
\node (2) at (0,0-6) {$W$};
\draw[->] (-2) --node[above ]{} (2);
\node (1) at (.8,.6-6) {$B_1$};
\node (1) at (0.7,-.6-6) {$B_2$};
\draw [->] (0.2,0.2-6) arc (-75:255:0.5)node[above]{};
\draw [->] (-0.2,-0.2-6) arc (105:425:0.5)node[below]{};
\node (3) at (1+2,0-6) {$W\oplus V$};
\node (1) at (1+2.8,.6-6) {};
\node (1) at (3.7,-.6-6) {};
\draw [->] (3.2,0.2-6) arc (-75:255:0.5)node[above]{};
\draw [->] (1+2-0.2,-0.2-6) arc (105:425:0.5)node[below]{};
\draw[->] (2) --node[above ]{} (3);
\node (4) at (3+1+2,0-6) {$V$};
\draw[->] (4) --node[above ]{} (-1);
\node (1) at (1+3+2.8,.6-6) {$C_1$};
\node (1) at (6.7,-.6-6) {$C_2$};
\draw [->] (6.2,0.2-6) arc (-75:255:0.5)node[above]{};
\draw [->] (3+1+2-0.2,-0.2-6) arc (105:425:0.5)node[below]{};
\draw[->] (3) --node[above ]{} (4);
\end{tikzpicture}
\] {Notice that the commutativity can be preserved by $\psi$. We have} \[
(1,1)\psi(\xi_{1}\oplus \xi_{2})\begin{pmatrix}
1\\1
\end{pmatrix}=\psi((1,1)(\xi_{1}\oplus \xi_{2})\begin{pmatrix}
1\\1
\end{pmatrix}).
\]

{Finally, $\psi$ is invertible.  In fact, an element $\zeta$} in $\Ext^1_\mathcal{REP}((V,C'_1,C'_2),(W,B'_1,B'_2))$ is an short exact sequence
\[\begin{tikzpicture}
\node (-2) at (-3,0) {$\zeta$:};
\node (-2) at (-2,0) {0};
\node (-1) at (8,0) {0.};
\node (2) at (0,0) {$W$};
\draw[->] (-2) --node[above ]{} (2);
\node (1) at (.8,.6) {$B'_1$};
\node (1) at (0.7,-.6) {$B'_2$};
\draw [->] (0.2,0.2) arc (-75:255:0.5)node[above]{};
\draw [->] (-0.2,-0.2) arc (105:425:0.5)node[below]{};
\node (3) at (1+2,0) {$W\oplus V$};
\node (1) at (1+2.8,.6) {$E'_1$};
\node (1) at (3.7,-.6) {$E'_2$};
\draw [->] (3.2,0.2) arc (-75:255:0.5)node[above]{};
\draw [->] (1+2-0.2,-0.2) arc (105:425:0.5)node[below]{};
\draw[->] (2) --node[above ]{$f'$} (3);
\node (4) at (3+1+2,0) {$V$};
\draw[->] (4) --node[above ]{} (-1);
\node (1) at (1+3+2.8,.6) {$C’_1$};
\node (1) at (6.7,-.6) {$C’_2$};
\draw [->] (6.2,0.2) arc (-75:255:0.5)node[above]{};
\draw [->] (3+1+2-0.2,-0.2) arc (105:425:0.5)node[below]{};
\draw[->] (3) --node[above ]{$g'$} (4);
\end{tikzpicture}
\] {We can construct a corresponding \begin{align*}
\varphi: \Ext^1_\mathcal{REP}((V,C'_1,C'_2),(W,B'_1,B'_2))&\rightarrow \Ext^1_\mathcal{REP}((V,C_1,C_2),(W,B_1,B_2)), \\&\zeta\mapsto\xi,
\end{align*} where $\xi$ is the short exact sequence represented as follow} \[
\begin{tikzpicture}
\node (-2) at (-3,0) {$\xi$:};
\node (-2) at (-2,0) {0};
\node (-1) at (8,0) {0.};
\node (2) at (0,0) {$W$};
\draw[->] (-2) --node[above ]{} (2);
\node (1) at (.8,.6) {$B’_1-\lambda \mathrm{Id}$};
\node (1) at (0.7,-.6) {$B’_2-\mu \mathrm{Id}$};
\draw [->] (0.2,0.2) arc (-75:255:0.5)node[above]{};
\draw [->] (-0.2,-0.2) arc (105:425:0.5)node[below]{};
\node (3) at (1+2,0) {$W\oplus V$};
\node (1) at (1+2.8,.6) {$E'_1-\lambda \mathrm{Id}$};
\node (1) at (3.7,-.6) {$E'_2-\mu \mathrm{Id}$};
\draw [->] (3.2,0.2) arc (-75:255:0.5)node[above]{};
\draw [->] (1+2-0.2,-0.2) arc (105:425:0.5)node[below]{};
\draw[->] (2) --node[above ]{$f'$} (3);
\node (4) at (3+1+2,0) {$V$};
\draw[->] (4) --node[above ]{} (-1);
\node (1) at (1+3+2.8,.6) {$C’_1-\lambda \mathrm{Id}$};
\node (1) at (6.7,-.6) {$C’_2-\mu \mathrm{Id}$};
\draw [->] (6.2,0.2) arc (-75:255:0.5)node[above]{};
\draw [->] (3+1+2-0.2,-0.2) arc (105:425:0.5)node[below]{};
\draw[->] (3) --node[above ]{$g'$} (4);
\end{tikzpicture}
\]
It is easy to verify that $\varphi$ is the inverse of $\psi$.
Therefore, $\psi$ is an isomorphism of $\Bbbk$-linear spaces.
\end{proof}

\begin{lemma}
\label{lem7.3}
For two bricks $(\Bbbk,\lambda_1,\lambda_2),(\Bbbk,\mu_1,\mu_2)\in \mathcal{REP}$, there is a nontrivial extension between them if and only if $\lambda_1=\mu_1,\lambda_2=\mu_2$. In this case ,we have\[
\dim \Ext_\mathcal{REP}^1((\Bbbk,\lambda_1,\lambda_2),(\Bbbk,\lambda_1,\lambda_2))=2.
\]
\end{lemma}

\begin{proof}
By Lemma \ref{lem7.2}, we can assume $\lambda_1=\lambda_2=0$. {If $\Ext_\mathcal{REP}^1((\Bbbk,0,0),(\Bbbk,\mu_1,\mu_2))$ is non-zero, then there exists $0\neq \xi$ in $\Ext_\mathcal{REP}^1((\Bbbk,0,0),(\Bbbk,\mu_1,\mu_2))$. By choosing a suitable basis of $\Bbbk$-linear spaces, $\xi$ can be represented as} the following commutative diagram\[
\begin{tikzpicture}
\node (2) at (-3,0) {$\xi:$};
\node (2) at (0,0) {$\Bbbk$};
\node (1) at (.8,.6) {$\mu_1$};
\node (1) at (0.7,-.6) {$\mu_2$};
\draw [->] (0.2,0.2) arc (-75:255:0.5)node[above]{};
\draw [->] (-0.2,-0.2) arc (105:425:0.5)node[below]{};
\node (3) at (1+2,0) {$\Bbbk^2$};
\node (1) at (1+2.8,.6) {$E_1$};
\node (1) at (3.7,-.6) {$E_2$};
\draw [->] (3.2,0.2) arc (-75:255:0.5)node[above]{};
\draw [->] (1+2-0.2,-0.2) arc (105:425:0.5)node[below]{};
\draw[->] (2) --node[above ]{$(1,0)^T$} (3);
\node (4) at (3+1+2,0) {$\Bbbk$};
\node (1) at (1+3+2.8,.6) {$0$};
\node (1) at (6.7,-.6) {$0$};
\draw [->] (6.2,0.2) arc (-75:255:0.5)node[above]{};
\draw [->] (3+1+2-0.2,-0.2) arc (105:425:0.5)node[below]{};
\draw[->] (3) --node[above ]{$(0,1)$} (4);
\node (-2) at (-2,0) {0};
\node (-1) at (8,0) {0};
\draw[->] (-2) --node[above ]{} (2);
\draw[->] (4) --node[above ]{} (-1);
\end{tikzpicture}
\]which implies that there exist $\mu_3,\mu_4\in\Bbbk$ such that $E_1=\begin{pmatrix}
    \mu_1&\mu_3\\0&0
\end{pmatrix}$ and $E_2=\begin{pmatrix}
    \mu_2&\mu_4\\0&0
\end{pmatrix}$.  Since $E_1E_2=E_2E_1$, we have $\mu_4\mu_1=\mu_2\mu_3$.

{If $\mu_1\ne0$, then $\mu_4=\mu_2\mu_3/\mu_1$ and $\begin{pmatrix}
    1&\mu_3/\mu_1\\0 &1
\end{pmatrix}: \Bbbk^2\rightarrow \Bbbk^2$ is an isomorphism. We have the following commutative diagram \[
\begin{tikzpicture}
\node (2) at (0,0) {$\Bbbk$};
\node (1) at (.8,.6) {$\mu_1$};
\node (1) at (0.7,-.6) {$\mu_2$};
\draw [->] (0.2,0.2) arc (-75:255:0.5)node[above]{};
\draw [->] (-0.2,-0.2) arc (105:425:0.5)node[below]{};
\node (3) at (1+2,0) {$\Bbbk^2$};
\node (1) at (1+2.8+.3,.8+.6) {$\begin{pmatrix}
	\mu_1&\mu_3\\0&0
	\end{pmatrix}$};
\node (1) at (4.2,-.8) {$\begin{pmatrix}
	\mu_2&\mu_2\mu_3/\mu_1\\0&0
	\end{pmatrix}$};
\draw [->] (3.2,0.2) arc (-75:255:0.5)node[above]{};
\draw [->] (1+2-0.2,-0.2) arc (105:425:0.5)node[below]{};
\draw[->] (2) --node[above ]{$(1,0)^T$} (3);
\node (4) at (3+1+2,0) {$\Bbbk$};
\node (1) at (1+3+2.8,.6) {$0$};
\node (1) at (6.7,-.6) {$0$};
\draw [->] (6.2,0.2) arc (-75:255:0.5)node[above]{};
\draw [->] (3+1+2-0.2,-0.2) arc (105:425:0.5)node[below]{};
\draw[->] (3) --node[above ]{$(0,1)$} (4);
\node (-2) at (-2,0) {0};
\node (-1) at (8,0) {0};
\draw[->] (-2) --node[above ]{} (2);
\draw[->] (4) --node[above ]{} (-1);

\node (2) at (0,0-3) {$\Bbbk$};
\node (1) at (.8,.6-3) {$\mu_1$};
\node (1) at (0.7,-.6-3) {$\mu_2$};
\draw [->] (0.2,0.2-3) arc (-75:255:0.5)node[above]{};
\draw [->] (-0.2,-0.2-3) arc (105:425:0.5)node[below]{};
\node (3) at (1+2,0-3) {$\Bbbk^2$};
\node (1) at (1+2+1.1,1.3-3.1) {$\begin{pmatrix}
    \mu_1&0\\0&0
\end{pmatrix}$};
\node (1) at (3.7,-1.6-3) {$\begin{pmatrix}
    \mu_2&0\\0&0
\end{pmatrix}$};
\draw [->] (3.2,0.2-3) arc (-75:255:0.5)node[above]{};
\draw [->] (1+2-0.2,-0.2-3) arc (105:425:0.5)node[below]{};
\draw[->] (2) --node[above ]{$(1,0)^T$} (3);
\node (4) at (3+1+2,0-3) {$\Bbbk$};
\node (1) at (1+3+2.8,.6-3) {$0$};
\node (1) at (6.7,-.6-3) {$0$.};
\draw [->] (6.2,0.2-3) arc (-75:255:0.5)node[above]{};
\draw [->] (3+1+2-0.2,-0.2-3) arc (105:425:0.5)node[below]{};
\draw[->] (3) --node[above ]{$(0,1)$} (4);
\node (-2) at (-2,0-3) {0};
\node (-1) at (8,0-3) {0};
\draw[->] (-2) --node[above ]{} (2);
\draw[->] (4) --node[above ]{} (-1);
\draw[-] (0,-.3)--(0,-2.7);
\draw[-] (0.1,-.3)--(0.1,-2.7);
\draw[-] (0+6,-.3)--(0+6,-2.7);
\draw[-] (0.1+6,-.3)--(0.1+6,-2.7);
\draw[->] (0+3,-.3)--node[left]{$\begin{pmatrix}
	1&{\mu_3}/{\mu_1}\\0&1
	\end{pmatrix}$}(0+3,-2.7);
\end{tikzpicture}
\]It follows that $\xi=0$, which is a contradiction to $\xi\neq0$. Hence $\mu_1=0$. Similarly, we can prove that $\mu_2=0$.

On the other hand, if $\mu_1=0$ and $\mu_2=0$, then by Lemma \ref{lem7.2} we have $$\Ext_\mathcal{REP}^1((\Bbbk,\lambda_1,\lambda_2),(\Bbbk,\lambda_1,\lambda_2))\cong \Ext_\mathcal{REP}^1((\Bbbk,0,0),(\Bbbk,0,0))\cong \Bbbk^2.$$}
\end{proof}

\begin{theorem}
We have\[
\fpdim(\mathcal{REP})=2.
\]
\end{theorem}

\begin{proof}
The conclusion is directly obtained by Lemma \ref{lem7.1} and Lemma \ref{lem7.3}.
\end{proof}

\begin{corollary}
For the representation category $\mathcal{REP}_r=\mathrm{rep}\ \Bbbk[x_1,x_2,\cdots,x_r]$, we have\[
\fpdim(\mathcal{REP}_r)=r.
\]
\end{corollary}

{

\begin{remark}
The above result shows that there also exist infinite dimensional algebras whose Frobenius-Perron dimension is equal to the maximal number of loops. So the further study on the characteristics of these infinite dimensional algebras are meaningful and challenging.
\end{remark}}

\noindent {\bf Acknowledgements.}\quad
This work is supported by the National Natural Science Foundation of China (Grant Nos. 11971398, 12131018 and 12161141001).



\vskip 5pt
\noindent {\tiny  \noindent Jianmin Chen, Jiayi Chen\\
School of Mathematical Sciences, \\
Xiamen University, Xiamen, 361005, Fujian, PR China.\\
E-mails: chenjianmin@xmu.edu.cn, jiayichen.xmu@foxmail.com}
\vskip 3pt

\end{document}